\documentclass[11pt,twoside]{amsart}
\usepackage{amsmath,amsfonts,amscd,amsthm}
\usepackage{amssymb}
\usepackage[flushmargin]{footmisc}
\usepackage{color}                    
\usepackage{esvect}

\newtheorem{remark}{Remark}
\newcommand{\dbar}{\ensuremath{\overline\partial}}
\newcommand{\dbarstar}{\ensuremath{\overline\partial^*}}
\newcommand{\C}{\ensuremath{\mathbb{C}}}

\newcommand{\R}{\ensuremath{\mathbb{R}}}

\DeclareFontFamily{U}{mathx}{\hyphenchar\font45}
\DeclareFontShape{U}{mathx}{m}{n}{
      <5> <6> <7> <8> <9> <10>
      <10.95> <12> <14.4> <17.28> <20.74> <24.88>
      mathx10
      }{}
\DeclareSymbolFont{mathx}{U}{mathx}{m}{n}
\DeclareFontSubstitution{U}{mathx}{m}{n}
\DeclareMathAccent{\widecheck}{0}{mathx}{"71}
\DeclareMathAccent{\wideparen}{0}{mathx}{"75}

\def\dom{\operatorname{Dom}\,}

\def\eps{\varepsilon}
\def\tf{\wt{f}}
\def\hf{\widehat{f}}
\def\cf{\widecheck{f}}

\def\range{\operatorname{\mathcal R}\,}

\def\omz{\Omega} 
\def\va{\vartheta}

\makeatletter
\newcommand{\sumprime}{\if@display\sideset{}{'}\sum%
	\else\sum'\fi}
\makeatother

\newtheorem{thm}{Theorem}[section]

\newtheorem{lem}[thm]{Lemma}
\newtheorem{cor}[thm]{Corollary}

\theoremstyle{definition}

\theoremstyle{remark}

\numberwithin{equation}{section}

\providecommand\ufootnote[1]{{\let\thefootnote\relax\footnote[0]{#1}}}

\newcommand{\dc}{\mathcal D}
\newcommand{\ec}{\mathcal E}

\newcommand{\nb}{\mathbb N}

\newcommand{\B}{\mathbb B}

\newcommand{\hb}{\mathbb H}

\newcommand{\N}{\mathbb{N}}

\newcommand{\ol}{\overline}

\newcommand{\wt}{\widetilde}

\DeclareMathOperator{\supp}{supp} 
 
\DeclareMathOperator{\rea}{Re}

\DeclareMathOperator{\dist}{dist} 
 \DeclareMathOperator{\Dom}{Dom}

\begin{document}

\title[Spectral Stability of the $\dbar$-Neumann Laplacian]{Spectral Stability of the $\dbar-$Neumann Laplacian: Domain Perturbations}

\author{Siqi Fu and Weixia Zhu}

\address{Department of Mathematical Sciences,
	Rutgers University, Camden, NJ 08102, USA} \email{sfu@rutgers.edu}

\address{School of Mathematical Sciences,
   Xiamen University, Xiamen, Fujian 361005, CHN}\email{zhuvixia@stu.xmu.edu.cn}

\begin{abstract} We study spectral stability of the $\bar\partial$-Neumann Laplacian on a bounded domain in $\C^n$ when the underlying domain is perturbed. In particular, we establish upper semi-continuity properties for the variational eigenvalues of the $\bar\partial$-Neumann Laplacian on bounded pseudoconvex domains in $\C^n$, lower semi-continuity properties on pseudoconvex domains that satisfy property ($P$), and quantitative estimates on smooth bounded pseudoconvex domains of finite D'Angelo type in $\C^n$.
	
\bigskip


	\smallskip
	
\end{abstract}

\maketitle

\tableofcontents

\section{Introduction}\label{sec:intro}

The $\dbar$-Neumann Laplacian is a prototype of an elliptic operator with non-coercive
boundary conditions~(\cite{KohnNirenberg65}). Since the fundamental work of Kohn \cite{Kohn63}
and H\"{o}rmander \cite{Hormander65}, it has been known that existence and regularity of the $\dbar$-Neumann Laplacian are closely connected to the boundary geometry of the underlying domains (see, e.g.,  \cite{FollandKohn72, BoasStraube99, DangeloKohn99, ChenShaw99, Straube10} for expositions on the subject). Spectral behavior of the $\dbar$-Neumann Laplacian has also been shown to be sensitive to the geometry of the domains. Positivity of the $\dbar$-Neumann can be used to characterize pseudoconvexity (see \cite{Fu08, FLS17} and references therein).  Spectral discreteness of the $\dbar$-Neumann Laplacian can be used to determine whether the boundary of a convex domain in $\C^n$ contains a complex variety (\cite{FuStraube98, FuStraube01}) and whether the boundary of a smooth bounded pseudoconvex Hartogs domain in $\C^2$ satisfies property ($P$), a potential theoretic property introduced by Catlin~\cite{Catlin84} (see \cite{FuStraube02, ChristFu05}).  Asymptotic behavior of the eigenvalues can be used to establish whether a smooth bounded pseudoconvex domain in $\C^2$ is of finite type (\cite{Fu08}). 

In physical sciences, exact values of quantities are oftentimes difficult--in some cases, impossible--to obtain and approximate values are observed and utilized instead. It is thus important to study how these quantities are affected when there are small perturbations of other parameters.  Spectral stability of the classical Dirichlet and Neumann Laplacians on domains in $\R^n$ has been studied extensively in literatures (see, e.g., \cite{Fuglede99, Davies00, BL07} and references therein).  In this paper, we study spectral stability of the $\dbar$-Neumann Laplacian on a bounded domain $\Omega$ in $\C^n$ when the underlying domain is perturbed.  There are several ways to measure spectral stability. Our focus here is on the variational eigenvalues.  The $k^{\rm th}$-variational eigenvalue $\lambda^q_k(\Omega)$ of the $\dbar$-Neumann Laplacian $\Box$ on $(0, q)$-forms ($1\le q\le n-1$) on $\Omega$ are defined through the min-max principle and they are bona fide eigenvalues when the spectrum is discrete (see Section~\ref{sec:prelim} below). We first establish the following upper semi-continuity property of the variational eigenvalues of the $\dbar$-Neumann Laplacian on a pseudoconvex domain.

\begin{thm}\label{main1}
	Let	$\omz_1$ be a bounded pseudoconvex domain in $\C^n$ with $C^1$-smooth boundary. Let $k$ be a positive integer. For any $\eps>0$, there exists $\delta>0$ such that for any pseudoconvex domain $\Omega_2$,
	\begin{equation}
	\lambda^q_k(\Omega_2)\le \lambda^q_k(\Omega_1)+\eps, \quad 1\le q\le n-1,
	\end{equation}	
	provided $d_H(\Omega_1, \Omega_2)<\delta$, where $d_H$ denotes the Hausdorff distance between the domains.
\end{thm}

Spectral theory of the $\dbar$-Neumann Laplacian differs substantially from that of the classical Laplacians because of  the non-coercive nature of the $\dbar$-Neumann boundary conditions.  Unlike the classical Dirichlet or Neumann Laplacian, spectral discreteness of the $\dbar$-Neumann Laplacian on a bounded domain $\Omega$ in $\C^n$ depends not only on the smoothness of the boundary but more importantly on geometric and potential properties of the boundary.  One difficulty in studying spectral stability of the $\dbar$-Neumann Laplacian is due to the fact that unlike the classical Neumann Laplacian, the restriction $f\vert_{\widehat\Omega}$ of a form $f\in\Dom(Q_\Omega)$, the domain of definition of the quadratic form associated with the $\dbar$-Neuman Laplacian on $\Omega$, need not belong to $\Dom(Q_{\widehat\Omega})$, where $\widehat\Omega$ is a subdomain of $\Omega$. Additionally, unlike the Dirichlet Laplacian, the extension of $f$ to zero outside of $\Omega$ does not make it belong to $\Dom(Q_{\widetilde\Omega})$ for a larger domain $\widetilde\Omega$.  To overcome these difficulties, we decompose a form in $\Dom(Q_\Omega)$ into tangential and normal components and treat them separately. Roughly speaking, the tangential component is treated as in the case of the Neumann Laplacian and the normal component is treated as in the case of the Dirichlet Laplacian.  

To establish the lower semi-continuity property of the variational eigenvalues, we will have to assume that the targeted domain satisfies property ($P$). Property ($P$) is a potential theoretic property introduced by Catlin \cite{Catlin84} to study compactness in the $\dbar$-Neumann problem. Kohn and Nirenberg \cite{KohnNirenberg65} showed that compactness of the $\dbar$-Green operator, the inverse of $\dbar$-Neuman Laplacian, implies exact global regularity of the $\dbar$-Neumann Laplacian.  (Compactness of the $\dbar$-Green operator is equivalent to spectral discreteness of the $\dbar$-Neumann Laplacian.)  Catlin showed that for  a bounded pseudoconvex domain with smooth boundary  in $\C^n$, property ($P$) implies compactness of the $\dbar$-Green operator. Straube showed that Catlin's theorem holds without the boundary smoothness assumption~(\cite{Straube97}). It remains an open problem whether or not the converse to Catlin's theorem is also true. 

\begin{thm}\label{main2}
	Let	$\omz_1$ be a bounded pseudoconvex domain in $\C^n$ with $C^1$  boundary that satisfies Property $(P_{q-1})$, $2\le q\le n-1$. Let $k$ be a positive integer. For any $\eps>0$, there exists $\delta>0$ such that for any pseudoconvex domain $\Omega_2$ whose $\dbar$-Neumann Laplacian has discrete spectrum on $(0, q)$-forms, 
	\begin{equation}
	\lambda^q_k(\Omega_2)\ge \lambda^q_k(\Omega_1)-\eps
	\end{equation}	
	provided $d_H(\Omega_1, \Omega_2)<\delta$.
\end{thm}

To establish the quantitative estimates, we further assume that the domains are of finite type. A notion of finite type was introduced by Kohn for smooth bounded pseudoconvex domains in $\C^2$ in connection with subellipticity of the $\dbar$-Neumann Laplacian~\cite{Kohn72}.   For domains in higher dimensions, a new finite type notion was introduced by D'Angelo~\cite{Dangelo82}: A smooth bounded domain in $\C^n$ is of finite type in the sense of D'Angelo if the normalized order of contact of complex analytic varieties with the boundary is finite. Catlin showed that for a smooth bounded pseudoconvex domain in $\C^n$, subellipticity of the $\dbar$-Neumann Laplacian is equivalent to the finite D'Angelo type~\cite{Catlin83, Catlin87}. Here we study spectral stability of the $\dbar$-Neumann Laplacian on such domains. Our main result in this regard is:

\begin{thm}\label{main3} 
	Let $\omz_j$ and $\omz$ be smooth bounded pseudoconvex domains in $\C^n$. Suppose $\omz_j$ and $\omz$ are of uniform finite $D_q$-type in $\C^n$, $1\le q\le n-1$. Let $k$ be a positive integer. Then there exist constants $\delta>0$ and $C_k>0$ such that
	\begin{equation}
		|\lambda^q_k(\omz_j)-\lambda^q_k(\omz)|\le C_k \delta_j,
	\end{equation}
	provided $\delta_j=d_H(\Omega, \Omega_j)<\delta$. 
\end{thm}
We refer the reader to Section~\ref{sec:finite-type} for precise definition of {\it uniform} finite type. Our analysis is based on Catlin's construction of bounded plurisubharmonic functions with large complex Hessians. We will also use a version of sharp Hardy inequality due to Brezis and Marcus~\cite{BrezisMarcus97} and an idea from Davies~\cite{Davies00}.

This paper is organized as follows. In Section~\ref{sec:prelim}, we recall the spectral theoretic
setup of the $\dbar$-Neumann Laplacian and relevant facts regarding the variational eigenvalues. In Section~\ref{sec:upper-semi}, we establish upper semi-continuity property for the variational eigenvalues of the $\dbar$-Neumann Laplacian on bounded pseudoconvex domains in $\C^n$ and prove Theorem~\ref{main1}. In Section~\ref{sec:lower-semi}, we study lower semi-continuity of the variational eigenvalues on bounded pseudoconvex domains satisfying property ($P$) and establish Theorem~\ref{main2}. In Section~\ref{sec:finite-type}, we obtain quantitative estimates, including Theorem~\ref{main3}, for stability of the variational eigenvalues on pseudoconvex domains of finite type. Section~\ref{sec:re} contains further results on convergence of the $\dbar$-Neumann  Laplacian in resolvent sense.

\section{Preliminary}\label{sec:prelim}
We first review relevant elements in general spectral theory.  Let $Q$ be a non-negative, densely defined, and closed
sesquilinear form on a complex Hilbert space $\hb$ with domain $\Dom(Q)$. Then $Q$
uniquely determines a non-negative self-adjoint operator $S$ such that
$\Dom(S^{1/2})=\Dom(Q)$ and
\[
Q(u, v)=\langle S^{1/2}u, \; S^{1/2}v\rangle
\]
for all $u, v\in \Dom(Q)$. Furthermore,
\[
\Dom(S)=\{ u\in\Dom(Q) \mid \exists f\in \hb, Q(u, v)=\langle f, v\rangle, \forall v\in \Dom(Q) \}.
\]
(See, e.g., Theorem 4.4.2 in \cite{Davies95}.) For any subspace
$L\subset\Dom(Q)$, let 
$$
\lambda_Q(L)=\sup\{Q(u, u) \mid  u\in L,
\|u\|=1\}.
$$
For any positive integer $k$, let
\begin{equation}\label{minmax}
\lambda_{k} (S)=\inf\{\lambda(L) \mid L\subset \Dom(Q),
\dim(L)=k\}
\end{equation}
be the $k^{\rm th}$ {\it variational} eigenvalues of $S$.
The resolvent set $\rho(S)$ of the operator $S$ consists of all $\lambda\in\C$ such that
$S-\lambda I\colon\Dom(S)\to \hb$ 
is both one-to-one and onto. It follows from the closed graph theorem that this operator has a bounded inverse, the resolvent operator $R_\lambda(S)=(S-\lambda I)^{-1}\colon \hb\to \Dom(S)$. The spectrum $\sigma(S)$ is the complement of $\rho(S)$ in $\C$. It is a non-empty closed subset of $[0, \ \infty)$. The lowest point in the spectrum is $\lambda_1(S)$.  The essential spectrum $\sigma_e(S)$ is the closed subset of $\sigma(S)$ that consists of isolated eigenvalues of infinite multiplicity and accumulation points of the spectrum. 
The bottom of the essential spectrum, $\inf\sigma_e(S)$, is the limit of $\lambda_k(S)$ as $k\to\infty$. The essential spectrum $\sigma_e(S)$ is empty if and only if $\lambda_{k}(S)\to\infty$ as $k\to\infty$. In this case, the variational eigenvalue $\lambda_{k}(S)$ is a {\it bona fide} eigenvalue of $S$. Indeed, it is the $k^{\text {th}}$ eigenvalue when the eigenvalues are arranged in increasing order and repeated according to multiplicity.  One approach to measuring spectral stability of a self-adjoint operator is through study how the variational eigenvalues vary as the operator is perturbed. The following simple lemma is well known~(compare \cite[Theorem~3.2]{BL07}):
 
 \begin{lem}\label{first} Let $S_i$, $i=1, 2$, be non-negative self-adjoint operators on Hilbert spaces $\hb_i$ with associated quadratic forms $Q_i$. Let $T\colon \Dom(Q_1)\to\Dom(Q_2)$ be a linear transformation from the domain of $Q_1$ to that of $Q_2$. Let $k$ be a positive integer. Suppose 
 there exist $0<\alpha_k<1/(2k)$ and $\beta_k>0$ such that for any orthonormal set $\{f_1, f_2, \ldots, f_k\}\subset\Dom(Q_1)$,  
 \begin{equation}\label{eq:approx}
 |\langle Tf_h, Tf_l\rangle_{2}-\delta_{hl}|\le \alpha_k \quad\text{and}\quad |Q_{2}(T f_h, T f_l)-Q_{1}(f_h,f_l)|\le \beta_k, \quad 1\le h, l\le k.
 \end{equation}
 	Then 
 	\begin{align}\label{30}
 		\lambda_k(S_2)\le \lambda_k(S_1)+2 k(\alpha_k\lambda_k(S_1)+\beta_k).
 	\end{align}
 \end{lem}
 
 \begin{proof} Let $L_k$ be any $k$-dimensional linear subspace of $\Dom(Q_1)$. Let $\{f_1, f_2, \ldots, f_k\}$ be an orthonormal basis
 for $L_k$.  For $f=\sum_{j=1}^k a_j f_j\in L_k$,   
it follows from the Cauchy-Schwarz inequality that
 	\begin{align*}
 		|Q_2(Tf, Tf)-Q_1(f, f)|^2&=\big|\sum_{h,l=1}^{k}\left(Q_2(T f_{h},T f_{l})-Q_1(f_{h},f_{l})\right)a_h\ol{a_l}\big|^2\\
 		&\le \sum_{h,l=1}^{k}\left|Q_{2}(T f_{h},Tf_{l})-Q_{1}(f_{h},f_{l})\right|^2\sum_{h,l=1}^{k}|a_h\ol{a_l}|^2\\
 		&\le k^2\beta_k^2 \|f\|_{1}^4.
 	\end{align*}
 	Thus 
 	\[
 	Q_2(Tf,Tf)\le Q_1(f, f)+k\beta_k \|f\|_{1}^2.
 	\]
 	Similarly, we have
 	\[
 	\|Tf\|^2_2\ge \|f\|^2_1-k\alpha_k \|f\|_1^2>\frac{1}{2} \|f\|_1^2.
 	\]
 	Therefore $T$ is one-to-one on $L_k$ and $T(L_k)$ is a $k$-dimensional linear subspace of $\Dom(Q_2)$. It follows from \eqref{minmax} that
 	\begin{equation}\label{29}
 		\begin{aligned}
 			\lambda_k(S_2)&\le\sup\bigg\{\frac{Q_2(Tf, Tf)}{\|T f\|^2_{2}} \mid f\in L_{k}\bigg\}\le \sup \bigg\{\frac{Q_1(f, f)+k\beta_k \|f\|_1^2}{(1-k\alpha_k)\|f\|_1^2} \mid f\in L_k\bigg\}\\
 			&\le  \dfrac{1}{1-k\alpha_k}\lambda_{Q_1}(L_k)+ \dfrac{k\beta_k}{1-k\alpha_k}=\lambda_{Q_1}(L_k)+\frac{k(\alpha_k\lambda_{Q_1}(L_k)+\beta_k)}{1-k\alpha_k} \\
 			&\le \lambda_{Q_1}(L_k)+2k(\alpha_k\lambda_{Q_1}(L_k)+\beta_k).
 		\end{aligned}
 	\end{equation}
 	Taking the infimum over all $k$-dimensional subspace $L_k$ in $\Dom(Q_1)$, we then obtain the desired inequality \eqref{30}.
 \end{proof}

\begin{remark}\label{remark1} Condition \eqref{eq:approx} in Lemma~\ref{first} can be replaced by the following: For any $k$-dimensional subspace $L_k$ of $\dom(Q_1)$ and  $f\in L_k$,
\begin{equation}\label{eq:approx2}
	\|Tf\|^2_2\ge (1-k\alpha_k) \|f\|_1^2 \quad {\text and} \quad
	Q_2(Tf,Tf)\le Q_1(f, f)+k\beta_k \|f\|_1^2.
\end{equation}
 This is easily seen from the proof above.
\end{remark}

We now recall a spectral theoretic setup for the $\bar\partial$-Neumann Laplacian. (We refer the readers to \cite{FollandKohn72, ChenShaw99, Straube10} for an in-depth treatment on regularity theory of the $\dbar$-Neumann Laplacian.) Let $L^2_{(0, q)}(\Omega)$ be the space of $(0, q)$-forms with $L^2$-coefficients on $\Omega$ with respect to the 
standard Euclidean metric.  Let $\dbar_q\colon L^2_{(0, q)}(\Omega)\to L^2_{(0, q+1)}(\Omega)$ be the maximally defined Cauchy-Riemann operator.  Thus $\Dom(\dbar_q)$, the domain of $\dbar_q$,  consists of those $u\in L^2_{(0, q)}(\Omega)$ such that $\dbar_q u$, defined in the sense of distribution, is in $L^2_{(0, q+1)}(\Omega)$. That is, there exists  $v\in L^2_{(0, q+1)}(\Omega)$
such that 
\[
\langle u, \vartheta \varphi\rangle =\langle v, \varphi\rangle 
\]
for all $\varphi\in \dc_{(0, q+1)}(\Omega)$, where $\vartheta$ is the formal adjoint of $\dbar_q$
and $\dc_{(0, q+1)}(\Omega)$ is the space of smooth $(0, q+1)$-forms with compact support in $\Omega$. Let $\dbarstar_q\colon L^2_{(0, q+1)}(\Omega)\to L^2_{(0, q)}(\Omega)$ be the adjoint of $\dbar_q$. Thus its domain is given by
	\begin{equation}\label{eq:dbarstar-dom}
		\Dom(\dbarstar_q)=\big\{u\in L^2_{(0, q+1)}(\Omega) \mid \exists C>0,
		|\langle u, \dbar_q v\rangle|\le C\|v\|,\ \forall v\in\Dom(\dbar_q)\big\}.
	\end{equation}
The maximally defined $\dbar_q$-operator can be regarded as the adjoint of the formal adjoint $\vartheta_q\colon L^2_{(0, q+1)}(\Omega)\to L^2_{(0, q)}(\Omega)$ whose domain $\Dom(\vartheta_q)=\dc_{(0, q+1)}(\Omega)$. The $\dbarstar_q$-operator is then the closure of $\vartheta_q$ and it is sometimes referred to as the {\it minimal extension} of $\vartheta_q$. Let $\Omega=\{z\in\C^n \mid \rho(z)<0\}$  be a bounded domain with a $C^1$-smooth defining function $\rho$ such that $|\nabla \rho|=1$ on $\partial\Omega$ and let 
$$
u=\sumprime_{|J|=q} u_J \,d\bar z_J \in C^1_{(0, q)}(\overline{\Omega}).
$$
Then $u\in\Dom(\dbarstar_{q-1})$ if and only if 
$$
(\dbar \rho)^*\lrcorner u=	\sumprime_{|K|=q-1}\left(\sum_{k=1}^n u_{kK}\frac{\partial \rho}{\partial z_k}\right) d\bar z_K=0
$$
on $\partial\Omega$, where
	\[
	(\dbar \rho)^*=\sum_{j=1}^n \frac{\partial \rho}{\partial z_j}\frac{\partial}{\partial \bar z_j}
	\]
is the dual $(0, 1)$-vector field of $\dbar \rho$ and $\lrcorner$ denotes the contraction operator.

For $1\le q\le n-1$, let
\[
Q_q(u, v)=\langle\dbar_q u, \dbar_q
v\rangle_\Omega+\langle\dbarstar_{q-1} u,
\dbarstar_{q-1} v\rangle_\Omega
\]
be the sesquilinear form on $L^2_{(0, q)}(\Omega)$ with domain
$\Dom(Q_{q})=\Dom(\dbar_q)\cap \Dom(\dbarstar_{q-1})$.
The self-adjoint operator $\square_{q}$ associated with
$Q_{q}$ is  called {\it the $\dbar$-Neumann Laplacian} on $L^2_{(0, q)}(\Omega)$.  
Consequently, $\square_q$ is given by
$$
\square_q=\dbar_{q-1}\dbarstar_{q-1}+\dbarstar_q\dbar_q
$$
with
$$	
\Dom(\square_q)=\{u\in L^2_{(0, q)}(\Omega) \mid u\in\Dom(Q_q), \dbar_q u\in \Dom(\dbarstar_{q}), \dbar_{q-1}^* u\in\Dom(\dbar_{q-1})\}.
$$
It is an elliptic operator with non-coercive boundary conditions \cite{KohnNirenberg65}.
We will use $\lambda^q_k(\Omega)$ to denote the $k^{th}$-variational eigenvalues of the $\dbar$-Neumann Laplacian $\Box$ on $(0, q)$-forms on $\Omega$, defined as above by 
 \begin{equation}\label{eq:minmax}
 \lambda^q_k(\Omega)=\inf_{L\subset\Dom(Q_q)\atop \dim L=k}\sup\limits_{u\in L\setminus\{0\}}\,Q_q(u,u)/\|u\|^2,
 \end{equation}
where the infimum takes over all linear subspace of $\Dom(Q_q)$ of dimension $k$. We will study spectral stability of the $\dbar$-Neumann Laplacian  as the underlying domain $\Omega$ is perturbed.  
There are several ways to study spectral stability of the $\dbar$-Neumann Laplacian. In this paper, we will focus on stability of the variational eigenvalues and the convergence in resolvent sense. Let $T_j$ and $T$ be self-adjoint operators on Hilbert space ${\mathbb H}$.  We say $T_j$ {\it converges to $T$ in norm} (respectively {\it strong}) {\it resolvent sense}  if for all $\lambda\in {\mathbb C}\setminus {\mathbb R}$, the resolvent operator $R_\lambda(T_j)=(T_j-\lambda I)^{-1}$ converges to $R_\lambda(T)= (T-\lambda I)^{-1}$ in norm (strongly).  It is well known that  if $T_j$ converges to $T$ in norm resolvent sense, then for any $\lambda\not\in\sigma(T)$,  $\lambda\not\in\sigma(T_j)$ for sufficiently large $j$, and if $T_j$ converges to $T$ in strong resolvent sense, then for any $\lambda\in\sigma(T)$, there exist $\lambda_j\in\sigma(T_j)$ so that $\lambda_j\to\lambda$. We refer the reader to \cite[\S VIII.7]{ReedSimon80} for relevant material. 

Perturbation of the domains will be measured by the Hausdorff distance.  Recall that for two
sets $A$ and $B$ in a metric space $(X, d)$, the Hausdorff distance between $A$ and $B$ is given by
\[
\tilde d_H(A,B)=\max\{\sup\limits_{x\in A}\inf\limits_{y\in B}d(x,y), \ \sup\limits_{y\in B}\inf\limits_{x\in A}d(x,y)\}.
\]
In this paper, we will measure the closeness between two domains $\Omega_1$ and $\Omega_2$ in $\C^n$ by the Hausdorff distance between them and their complements using the Euclidean metric. We set  
$$
d_H(\Omega_1, \Omega_2)=\max\{\tilde d_H(\Omega_1, \Omega_2), \ \tilde d_H(\Omega^c_1, \Omega^c_2)\}.
$$
For $\delta>0$, let
\[
\Omega^-_\delta=\{z\in\Omega \mid \dist(z, \Omega^c)>\delta\} \quad\text{and}\quad \Omega^+_\delta=\{z\in\C^n \mid \dist(z, \Omega)<\delta\}.
\]
It is easy to see that $d_H(\Omega_1, \Omega_2)<\delta$ if and only if
\[
\overline{(\Omega_2)^-_\delta}\subset\Omega_1\subset (\Omega_2)^+_\delta\quad\text{and}\quad
\overline{(\Omega_1)^-_\delta}\subset\Omega_2\subset (\Omega_1)^+_\delta.
\]

\section{Upper semi-continuity}\label{sec:upper-semi}

In this section, we establish several upper semi-continuity properties for the variational eigenvalues of the $\dbar$-Neumann Laplacian when the underlying domain is perturbed. We first study spectral stability of the $\dbar$-Neumann Laplacian when the underlying domain is exhausted by subdomains from inside. 

We will use $Q_{q, \Omega}$ to denote the quadratic form associated with the $\bar\partial$-Neumann Laplacian $\Box_{q, \Omega}$ acting on $(0, q)$-forms on $\Omega$. Let $\Omega_2\subset\Omega_1$ be bounded pseudoconvex domains in $\C^n$. Unlike the classical Neumann Laplacian, for a $(0, q)$-form $f\in\Dom(Q_{q, \Omega_1})$, its restriction to $\Omega_2$ is no longer in $\Dom(Q_{q, \Omega_2})$. The following regularization procedure was introduced by Straube~\cite{Straube97} (compare also \cite{MichelShaw01}) to overcome this difficulty: For $f\in\Dom(Q_{q, \Omega_1})$, 
we define
 \begin{align}\label{1}
 	T f=\dbarstar_{q, \Omega_2} N_{q+1, \Omega_2}(\dbar_{q, \Omega_1} f)\big\vert_{\Omega_2}+\dbar_{q-1, \Omega_2} N_{q-1, \Omega_2}(\dbarstar_{q-1, \Omega_1} f)\big\vert_{\Omega_2}.
 \end{align}
When $q=1$, $N_{0, \Omega_2}$ is the inverse of the restriction of $\Box_{0, \Omega_2}$ to the orthogonal complement $\ker(\dbar_{0, \Omega_2})^\perp=\range(\dbarstar_{0, \Omega_2})$ of the kernel of $\dbar_{0, \Omega_2}$ such that $\Box_{0, \Omega_2}N_{0, \Omega_2}=I-P_{0, \Omega_2}$, where $P_{0, \Omega_2}$ is the Bergman projection on $\Omega_2$ (see \cite[Theorem 4.4.3]{ChenShaw99}). Hereafter, for economy of notations, we will suppress the subscripts involving $q$ when doing this causes no confusion and instead use subscript $1$ and $2$ to indicate that the operators act on $\Omega_1$ and $\Omega_2$ respectively. Evidently, $T$ is a linear transformation from $\Dom(Q_1)$ into $\Dom(Q_2)$.  In light of Lemma~\ref{first}, in order to estimate the difference between variational eigenvalues on $\Omega_1$ and $\Omega_2$, we need to compare $f$ and $T f$. 

\begin{lem}\label{lm:T}  Let $\Omega_1$ be a bounded pseudoconvex domain in $\C^n$ and let $f\in\Dom(Q_1)$. For any $\eps>0$, there exists $\delta>0$ such that for any pseudoconvex domain $\Omega_2\subset\Omega_1$ with $d_H(\Omega_1, \Omega_2)<\delta$,
	\begin{equation}\label{eq:t1}
	\|f-Tf\|_{\Omega_2}+\|\dbar (f-Tf)\|_{\Omega_2}+\|\vartheta(f-Tf)\|_{\Omega_2}<\eps,
	\end{equation}
	where $\vartheta$ is the formal adjoint of $\dbar$.
\end{lem}	

\begin{proof} Since  $\dc_{(0,q)}(\omz_1)$ is dense in $\Dom(\dbarstar_1)$ in the graph norm $\|f\|_{\omz_1}+\|\dbarstar f\|_{\omz_1}$, for any $0<\varepsilon<1$, there exists  $\phi\in\dc_{(0,q)}(\omz_1)$ such that $\|f-\phi\|_{\omz_1}+\|\vartheta (f-\phi)\|_{\omz_1}<\varepsilon$. We choose $\delta$ sufficiently small such that $\supp\phi\subset\subset\Omega_2$. Thus $\phi\in\Dom(\dbarstar_2)$. It follows that
 \begin{equation}\label{4}
 	\begin{aligned}
 		Tf&= \dbarstar_2\dbar_2N_2 f+\dbar_2N_2\vartheta f \\
 		&=f-\dbar_2\dbarstar_2 N_2 f+\dbar_2N_2\vartheta f\\
 		&=f-\dbar_2\dbarstar_2N_2 (f-\phi)+\dbar_2N_2\vartheta (f-\phi).\\
 	\end{aligned}
 \end{equation}
Here we have used the orthogonal decomposition $u=\dbar\dbarstar N u +\dbarstar\dbar N u$  and commutative properties $N_2\dbar_2 =\dbar_2 N_2$ on $\Dom(\dbar_2)$ and $N_2\dbarstar_2=\dbarstar_2 N_2$ on $\Dom(\dbarstar_2)$.  Moreover, we have
 \begin{equation}\label{4b}
 \dbar_2 T f=\dbar f
 \end{equation}
and
 \begin{equation}\label{5}
 	\begin{aligned}
 		\dbarstar_2 Tf &=\dbarstar_2\dbar_2 N_2\vartheta f=\vartheta f-\dbar_2\dbarstar_2 N_2\vartheta f\\
 		&=\vartheta f-\dbar_2\dbarstar_2 N_2 \vartheta (f-\phi).
 	\end{aligned}
 \end{equation}
The desired inequality \eqref{eq:t1} then follows from H\"{o}rmander's $L^2$-estimates for the $\dbar$-operator which imply that $\dbar_2 N_2$ is a bounded operator whose norm is bounded from above by a constant depending only on the diameter of $\Omega_2$ (see, e.g., \cite[Theorem~4.4.1]{ChenShaw99}).
\end{proof}

We have the following upper semi-continuity property for the variational eigenvalues defined by \eqref{eq:minmax}.

\begin{thm}\label{prop1}
Let	$\omz_1$ be a bounded pseudoconvex domain in $\C^n$. Given $1\le q\le n-1$ and $k\in\N$. For any $\eps>0$, there exists $\delta>0$ such that for any pseudoconvex domain $\Omega_2\subset\Omega_1$, 
\begin{equation}\label{eq:s1}
	\lambda^q_k(\Omega_2)\le \lambda^q_k(\Omega_1)+\eps
\end{equation}	
provided $d_H(\Omega_1, \Omega_2)<\delta$.
\end{thm}

\begin{proof}
	For any $0<\tilde\varepsilon<1$, there exists a $k$-dimensional subspace $L_k\subset \Dom(Q_1)$ such that $\lambda_{Q_1}(L_k)=\sup\{Q_1(f,f)|\,f\in L_k, \|f\|_{\omz_1}=1\} \le\lambda_k(\omz_1)+\tilde\varepsilon$. (As before, we will drop the superscript $q$ for economy of notations when doing so causes no confusion.) Consider $Q_1(\cdot, \cdot)$ as a sesquilinear form on $L_k\times L_k$. Then there exists an orthonormal basis $\{f_1,\cdots,f_k\}$ of $L_k$ such that $Q_1(f_h,f_l)=\gamma_l\delta_{hl}$, $1\le h,l\le k$, and $0\le \gamma_1\le\cdots\le\gamma_k=\lambda_{Q_1}(L_k)$.  Note that 
\begin{equation}\label{t1a}
\|\dbar f_l\|_{\omz_1}^2+\|\dbarstar f_l\|_{\omz_1}^2\le\lambda_k(\omz_1)+\tilde\varepsilon.
\end{equation}
	Furthermore, by choosing $\delta$ sufficiently small, we can assume that 
\begin{equation}\label{t2}
\|f_l\|_{\omz_1\setminus\omz_2}+\|\dbar f_l\|_{\omz_1\setminus\omz_2}+\|\dbarstar f_l\|_{\omz_1\setminus\omz_2}\le \tilde\varepsilon
\end{equation} 
for all $1\le l\le k$.  We have		
	\begin{equation}\label{6}
	\begin{aligned}
&\left|\langle Tf_{h},Tf_{l}\rangle_{\omz_2}-\delta_{hl}\right|=\left|\langle Tf_{h},Tf_{l}\rangle_{\omz_2}-\langle f_h,f_l\rangle_{\omz_1}\right|\\ 
	&\qquad \le\left|\langle Tf_{h}-f_h,Tf_{l}\rangle_{\omz_2}\right|+\left|\langle f_h,Tf_{l}-f_l\rangle_{\omz_2}\right|+\big|\langle f_h,f_l\rangle_{\omz_1\setminus\omz_2}\big|\\ 
	&\qquad \le \|Tf_{h}-f_h\|_{\omz_2}\|Tf_{l}\|_{\omz_2}+\|f_h\|_{\omz_2}\|Tf_{l}-f_l\|_{\omz_2}+\|f_h\|_{\omz_1\setminus\omz_2}\|f_l\|_{\omz_1\setminus\omz_2}\\
	&\qquad\le C\tilde\varepsilon.
	\end{aligned}
	\end{equation}
	Since 	
	\begin{equation}\label{7}
	\begin{aligned}
		Q_{2}(Tf_{h},Tf_{l})-Q_1(f_h,f_l)&= \langle\dbar_2 Tf_{h},\dbar_2 Tf_{l}-\dbar f_l\rangle_{\omz_2}+\langle\dbar_2 Tf_{h}-\dbar f_h,\dbar f_{l}\rangle_{\omz_2} \\
		&-\langle\dbar f_{h},\dbar f_{l}\rangle_{\omz_1\backslash\omz_2}+\langle\dbarstar_2 Tf_{h},\dbarstar_2 Tf_{l}-\dbarstar f_l\rangle_{\omz_2}\\
		&+\langle\dbarstar_2  Tf_{h}-\dbarstar f_h,\dbarstar f_{l}\rangle_{\omz_2}-\langle\dbarstar f_{h},\dbarstar f_{l}\rangle_{\omz_1\backslash\omz_2},
	\end{aligned}
	\end{equation}
	it follows from \eqref{eq:t1}, \eqref{t1a}, \eqref{t2} and the Cauchy-Schwarz inequality that
	\begin{equation}
\left|Q_{2}(Tf_{h},  Tf_{l})-Q_1(f_h, f_l)\right|\le C(\lambda^{1/2}_k(\omz_1)+\tilde\varepsilon^{1/2})\tilde\varepsilon.
\end{equation}
By Lemma~\ref{first}, we have 
\[
\lambda_k(\Omega_2)\le \lambda_k(\Omega_1)+\eps,
\]
provided $\tilde\eps$ is sufficiently small. 	
\end{proof}	

As a direct consequence of Theorem~\ref{prop1}, we have:

\begin{cor}\label{cor1}
	Let	$\omz$, $\omz_j$ be bounded pseudoconvex domains in $\C^n$ such that $\omz_j\subset\omz$ and $d_H(\omz_j, \omz)\to 0$ as $j\to\infty$. Let $1\le q\le n-1$ and $k\in\N$. Then  
	\begin{equation}
		\limsup_{j\to\infty}\lambda^q_k(\Omega_j)\le \lambda^q_k(\Omega).
	\end{equation}
\end{cor}

We now study stability of variational eigenvalues of the $\dbar$-Neumann Laplacian on a bounded pseudoconvex domain $\Omega$ as it is encroached--not necessarily from inside--by pseudoconvex domains. Let $\Omega$ be a pseudoconvex domain in $\C^n$ with $C^1$-smooth boundary. Let $\rho(z)$ be the signed distance function such that $\rho(z)=-\dist(z, \partial\Omega)$ on $\Omega$ and $\rho(z)=\dist(z, \partial\Omega)$ on $\C^n\setminus\Omega$. Then $\rho$ is $C^1$ in a neighborhood $U$ of $\partial\Omega$ and $|\nabla\rho(z)|=1$ on $U$ (see \cite{KrantzParks81}). Let $z'\in\partial\omz$ and let $U'\subset U$ be a tubular neighborhood of $z'$ such that $|\nabla\rho(z)-\nabla\rho(z')|<1/2$ when $z\in U'$. Denote $\vv n(z)=\nabla r(z)$ and 
$$
\Omega^\pm_\delta=\{z\in\C^n \mid r(z)<\pm\delta\}.
$$ 
Shrinking $U'$ if necessary, then for sufficiently small $\delta>0$, we have $z-2\delta \vv n(z')\in \Omega$ for any $z\in U'\cap \Omega^+_\delta$ and $z+2\delta\vv n(z')\not\in\Omega$ for any $z\in U'\setminus\Omega^-_\delta$. Furthermore, 
	$$\dist(z- 2\delta\vv n(z'),\partial\omz)\ge\dist(z - 2\delta\vv n(z),\partial\omz)-2\delta|\vv n(z)-\vv n(z')|>2\delta-\delta=\delta$$
	for all $z\in U'\cap\Omega^+_\delta$. 
We choose a finite covering $\{U^l\}_{l=0}^{m}$ of $\ol{\omz}$ such that $U^0$ is relatively compact in $\omz$ and each $U^l$, $1\le l\le m,$ is a tubular neighborhood about some $z^l\in\partial\omz$ constructed as above. Write $\vv n^l=\vv n(z^l)$. We then have
$$
\bigcup_{l=1}^{m}\left\lbrace z-2\delta\vv n^l\,|\,z\in U^l\cap\omz\right\rbrace \bigcup U^0 \subset   \omz^-_\delta
$$
and 
$$
\bigcup_{l=1}^{m}\left\lbrace z+2\delta\vv n^l\,|\,z\in U^l\cap\omz\right\rbrace\bigcup U^0 \supset   \omz^+_\delta.
$$
Let $\{\psi^l\}_{l=0}^{m}$ be a partition of unity subordinated the covering $\{U^l, \ 0\le l\le m\}$ such that $\supp \psi^l\subset U^l$. 
Let $f\in \Dom(Q_{q, \Omega})$. Let $\wt f$ be the form obtained by extending $f$ to 0 outside of $\omz$. Then $\vartheta \widetilde{f}=\widetilde{\dbarstar_\Omega f}\in L^2_{(0, q-1)}(\C^n)$ (see \cite[p.~31]{ChenShaw99}). Let  
\begin{equation}\label{hat}
\widehat f_\delta(z)=\psi^0(z) f(z)+\sum_{l=1}^{m}\psi^l(z)f(z-2\delta\vv{n}^l)	
\end{equation}
for $z\in\omz^+_\delta$ and let
\begin{equation}\label{check}
		\widecheck f_\delta (z)=\psi^0(z)\tf(z)+\sum_{l=1}^{m}\psi^l(z)\tf(z+2\delta\vv{n}^l)
		\end{equation}
for $z\in\omz$. Here we use $f(z\pm 2\delta\vv{n}^l)$ to denote the form obtained by replacing the coefficient $f_J(z)$ of the form $f$ by $f_J(z\pm 2\delta\vv{n}^l)$:
\[
f(z\pm 2\delta\vv{n}^l)=\sumprime_{|J|=q} f_J(z\pm 2\delta\vv{n}^l) d\bar z_J.
\]
Notice that $\widecheck f_\delta(z)$ is supported on $\Omega^-_\delta$.
Roughly speaking, the form $\widehat f_\delta$ and $\widecheck f_\delta$ are respectively the {\it push-out} and {\it push-in} of $f$ along the normal direction by $\delta$ unit. These constructions are used to counter the fact that the restriction of $f$ to a subdomain does not necessarily belong to $\Dom(\dbarstar)$ on the subdomain and the extension of $f$ to zero outside of $\Omega$ does not necessarily belong to $\Dom(\dbar)$ on a larger domain. 
 
\begin{lem}\label{lm:hatcheck} Let $\Omega$ be a bounded domain in $\C^n$ with $C^1$-smooth boundary. Let $f\in \Dom(Q_{q, \Omega})$. Then
	$\widehat f_\delta\in \Dom(\dbar_{\Omega_\delta^+})$,  $\widecheck f_\delta\in\Dom(\dbarstar_{\Omega_\delta^-})$, and 
	\begin{equation}\label{eq:hatcheck}
	\|\widehat f_\delta\|_{\Omega^+_\delta}\le C\|f\|_\Omega, \quad \|\dbar\widehat f_\delta\|_{\Omega^+_\delta}\le C(\|f\|_\Omega+\|\dbar f\|_\Omega)
	\end{equation}
	and
	\begin{equation}\label{eq:hatcheck2}
	\|\widecheck{f}_\delta\|_{\Omega}\le C\|f\|_\Omega, \quad  \|\dbarstar\widecheck{f}_\delta\|_{\Omega}\le C(\|f\|_\Omega+\|\dbarstar f\|_\Omega)
	\end{equation}  
for some constant $C>0$ independent of $\delta$. Furthermore,
\begin{equation}\label{eq:hatcheck3}
	\|\widehat f_\delta -\widetilde f\|_{\Omega^+_\delta}+\|\dbar\widehat f_\delta -\widetilde{\dbar f}\|_{\Omega^+_\delta}\to 0
\end{equation}
and 
\begin{equation}\label{eq:hatcheck4}
	\|\widecheck f_\delta -f\|_{\Omega}+\|\vartheta\widecheck f_\delta -\vartheta f\|_{\Omega}\to 0
\end{equation}
as $\delta\to 0$, where $\widetilde{\dbar f}$, as before, is the extension of $\dbar f$ to $0$ outside of $\Omega$.
\end{lem}

\begin{proof} The first part of the lemma follows directly from the definitions of $\widehat f_\delta$ and $\widecheck f_\delta$. Notice that $\supp\widecheck f_\delta \subset \Omega^-_\delta$ and $\vartheta\widecheck f_\delta\in L^2_{(0, q)}(\C^n)$. Hence $\widecheck f_\delta\in\Dom(\dbarstar_{\Omega^-_\delta})$. Since
	
	\begin{equation}\label{hf}
		\begin{aligned}
			&\|\widehat f_\delta-\wt f\|_{\Omega^+_\delta}+\|\dbar \widehat f_\delta-\wt {\dbar f}\|_{\Omega_\delta^+}\\
			&\lesssim\sum_{l=1}^{m}\left\|f(z-2\delta\vv{n}^l)-\wt f(z)\right\|_{\omz_\delta^+\cap U^l}+\sum_{l=1}^{m}\left\|\dbar f(z-2\delta\vv{n}^l)-\wt{\dbar f}(z)\right\|_{\omz_\delta^+\cap U^l}\\
			&\le \sum_{l=1}^{m}\left\|\widetilde f(z-2\delta\vv{n}^l)-\widetilde f(z)\right\|_{\C^n}+\sum_{l=1}^{m}\left\|\widetilde{\dbar f}(z-2\delta\vv{n}^l)-\widetilde{\dbar f}(z)\right\|_{\C^n},
		\end{aligned}
	\end{equation}
	we then obtain \eqref{eq:hatcheck3} from the dominated convergence theorem. The proof of \eqref{eq:hatcheck4} is similar and is left to the reader. 
	\end{proof}	

\begin{thm}\label{prop5}
Let	$\omz_1$ be a bounded pseudoconvex domain in $\C^n$ with $C^1$ boundary. Let $1\le q\le n-1$ and $k\in\N$. For any $\eps>0$, there exists $\delta>0$ such that for any pseudoconvex domain $\Omega_2$, we have
\begin{equation}
	\lambda^q_k(\Omega_2)\le \lambda^q_k(\Omega_1)+\eps
\end{equation}	
provided $d_H(\Omega_1, \Omega_2)<\delta$. 
\end{thm}

\begin{proof} Since $d_H(\Omega_1, \Omega_2)<\delta$, we have $(\Omega_1)^-_\delta\subset \Omega_2\subset (\Omega_1)^+_\delta$. For $f\in \Dom(Q_1)$, let $\widetilde f$ be the form obtained by extending $f$ to $0$ outside $\Omega_1$ and let $\widehat f_\delta$ be
	the forms constructed by \eqref{hat} as above (with $\Omega$ replaced by $\Omega_1$). Let
	\begin{align}
	T_\delta f=\dbarstar_2 N_2\dbar_2 \widehat f_\delta+\dbar_2 N_2\va \tf.
	\end{align}
	Then $T_\delta f\in \Dom(Q_2)$. Furthermore,  for any $\phi\in\Dom(\dbarstar_2)$, we have
	\begin{equation}\label{Tf}
	\begin{aligned}
	T_\delta f&=\dbarstar_2\dbar_2N_2 \hf_\delta+\dbar_2 N_2\va (\tf-\phi)+\dbar_2 \dbarstar_2N_2\phi\\
	&=\dbarstar_2\dbar_2N_2 \tf+\dbarstar_2\dbar_2N_2 (\hf_\delta-\tf)+\dbar_2 N_2\va (\tf-\phi)+\dbar_2 \dbarstar_2N_2(\phi-\tf)+\dbar_2 \dbarstar_2N_2\tf\\
	&=\tf+\dbarstar_2\dbar_2N_2 (\hf_\delta-\tf)+\dbar_2 N_2\va (\tf-\phi)+\dbar_2 \dbarstar_2N_2(\phi-\tf).\\
	\end{aligned}
	\end{equation}
	Moreover,
	\begin{equation}\label{df1}
	\dbar_2 T_\delta f=\dbar_2\dbarstar_2N_2\dbar_2\hf_\delta=\dbar_2\hf_\delta
	\end{equation}
	and
	\begin{align}\label{df2}
	\dbarstar_2 T_\delta f=\dbarstar_2\dbar_2 N_2\va \tf=\va \tf-\dbar_2\dbarstar_2 N_2\va \tf=\va \tf-\dbar_2\dbarstar_2 N_2\va(\tf-\phi).
	\end{align}	
	
	Let $L_k$ be a $k$-dimensional subspace of $\Dom(Q_1)$ with an orthonormal basis  $\{f_1,\cdots,f_k\}$.  For any $0<\eps<1$, by choosing $\delta$ sufficiently small, we have that 
\begin{equation}
\sum_{l=1}^k \big( \|f_l\|^2_{\omz_1\setminus\omz_2}+\|\dbar f_l\|^2_{\omz_1\setminus\omz_2}+\|\dbarstar f_l\|^2_{\omz_1\setminus\omz_2}\big)< \varepsilon^2. 
\end{equation} 
Since $\dc_{(0,q)}(\omz_1)$ is dense in $\Dom(\dbarstar_1)$ in the graph norm $\|f\|_{\omz_1}+\|\dbarstar f\|_{\omz_1}$, there exists a $\phi_l\in\dc_{(0,q)}(\omz_1)$ such that 
\begin{equation}\label{phi-1}
\sum_{l=1}^k\big(\|f_l-\phi_l\|^2_{\omz_1}+\|\dbarstar f_l-\dbarstar\phi_l\|^2_{\omz_1}\big)<\varepsilon^2.
\end{equation}	
By choosing $\delta$ sufficiently small,  we have $\supp\phi_l\subset\subset\Omega_2$. Thus $\phi_l\in\Dom(\dbarstar_2)$.
Let 
$$
f=\sum_{l=1}^k c_l f_l \quad \text{and} \quad 
\phi=\sum_{l=1}^k c_l \phi_l.
$$ 
It follows from the Cauchy-Schwarz inequality that
\begin{equation}
\|f\|^2_{\omz_1\setminus\omz_2}+\|\dbar f\|^2_{\omz_1\setminus\omz_2}+\|\dbarstar f\|^2_{\omz_1\setminus\omz_2}\le k \varepsilon^2 \|f\|^2_{\Omega_1}.
\end{equation} 

From \eqref{Tf}, we have
\[
\|T_\delta f -\widetilde{f}\|_{\Omega_2}\le \|\widehat{f}_\delta -\widetilde{f}\|_{\Omega_2}+\|\widetilde{f}-\phi\|_{\Omega_2}+C\|\vartheta(\widetilde{f}-\phi)\|_{\Omega_2},
\]
where the constant $C$ depending only on the diameter of $\Omega_2$, which can be assumed to be uniformly bounded from above. It follows from Lemma~\ref{lm:hatcheck} and \eqref{phi-1} that
\begin{equation}\label{t1}
	\begin{aligned}
\big|\|T_\delta f\|^2_{\Omega_2} - \|f\|^2_{\Omega_1} \big|&=\big|\| T_\delta f\|^2_{\Omega_2}-\|\widetilde f\|^2_{\Omega_2}-\|f\|^2_{\Omega_1\setminus\Omega_2}\big| \\
&\le \big(\|T_\delta f\|_{\Omega_2}+\|\widetilde f\|_{\Omega_2}\big)\|T_\delta f-\widetilde f\|_{\Omega_2}+\|f\|^2_{\Omega_1\setminus\Omega_2}\\
&\le C\eps \|f\|^2_{\Omega_1}.
\end{aligned}
\end{equation}

From \eqref{df1}, \eqref{df2}, \eqref{phi-1} and Lemma~\ref{lm:hatcheck}, we have  
\[
\|\dbar_2 T_\delta f -\widetilde{\dbar f}\|_{\Omega_2}=\|\dbar \widehat{f}_\delta -\widetilde{\dbar f}\|_{\Omega_2}\le C\eps \|f\|_{\Omega_1} 
\]	
and
\[
\|\dbarstar_2 T_\delta f-\vartheta\widetilde f\|_{\Omega_2}\le \|\vartheta (\widetilde f-\phi)\|_{\Omega_2}\le C\eps \|f\|_{\Omega_1}.
\]	
Therefore, similar to \eqref{t1}, we have
\[
\big| Q_2(T_\delta f, T_\delta f)-Q_1(f, f)\big|\le C\eps \|f\|^2_{\Omega_1}.
\]
By Lemma~\ref{first} and the subsequent remark, we then have 
$$\lambda_k(\omz_2)\le \lambda_k(\Omega_1)+C\varepsilon.$$
\end{proof}	

As a direct consequence of Theorem~\ref{prop5}, we have:
\begin{cor}\label{cor2}
	Let	$\omz$, $\omz_j$ be bounded pseudoconvex domains in $\C^n$ such that $\partial\Omega$ is $C^1$ and $d_H(\omz_j, \omz)\to 0$ as $j\to\infty$. Then for any $k\in\nb$, 
	\begin{equation}
		\limsup_{j\to\infty}\lambda^q_k(\Omega_j)\le \lambda^q_k(\Omega), \quad 1\le q\le n-1.
	\end{equation}
\end{cor}

The upper semi-continuity property of the variational eigenvalues also holds without the pseudoconvexity assumption when restricted to level sets.

\begin{thm}\label{prop3} Let $\Omega=\{z\in\C^n \mid \rho<0\}$ be a bounded domain in $\C^n$ with $C^2$-smooth boundary where $\rho\in C^2$ is a defining function of $\Omega$ with $|\nabla\rho|=1$ on $\partial\Omega$. For $\delta>0$, let $\Omega^-_\delta=\{z\in\Omega \mid \rho^-_\delta=\rho+\delta<0\}$
and $\Omega^+_\delta=\{z\in\C^n \mid \rho^+_\delta=\rho-\delta<0\}$. Then for any $k\in\nb$,
\[
\limsup_{\delta\to 0^+} \lambda_k(\Omega^{\pm}_\delta)\le \lambda_k(\Omega).
\]
\end{thm}	
\begin{proof} Since $C^1_{(0, q)}(\ol{\Omega})\cap\Dom(\dbarstar_\Omega)$ is dense in $\Dom(Q_\Omega)$ in the graph norm $(\|f\|^2_\Omega +Q_\Omega(f, f))^{1/2}$ (see, e.g., \cite[Lemma~4.3.2]{ChenShaw99}), it is sufficient to work on forms in $C^1_{(0, q)}(\ol{\Omega})\cap\Dom(\dbarstar_\Omega)$. Let $f\in C^1_{(0,q)}(\ol{\omz})\cap\Dom(\dbarstar_\Omega)$.
	Write $$f^\nu=f_{N}\wedge\dbar\rho,\quad f^\tau=f-f^\nu,$$
	where
	$$f_N:=(\dbar\rho)^*\lrcorner f=\sumprime_{|K|=q-1}\left(\sum_{j=1}^{n}\frac{\partial\rho}{\partial z_j}f_{jK}\right)d\bar{z}_K.
	$$ 
	Notice that $f_N=0$ on $\partial\Omega$ and $(\dbar\rho)^*\lrcorner f^\tau=0$ on $\Omega$. We extend $f_N$ to be 0 outside of $\Omega$. 
	Let $\{U^l\}_{l=0}^m$ be an open covering of $\ol{\Omega}$ and let $\{\psi^l\}_{l=0}^m$ be a partition of unity subordinated to the covering as in the setup preceding Lemma~\ref{lm:hatcheck}.  Set
	\begin{align}
			&f_{\delta}^{-\nu} (z)=\psi^0(z)f_N(z)\wedge\dbar\rho+\sum_{l=1}^{m}\psi^l(z)f_N(z+2\delta\vv{n}^l)\wedge\dbar\rho \notag\\
	\intertext{and}
			&f_{\delta}^{-\tau}(z)=f^\tau(z). \notag
		\end{align}
for $z\in\omz^-_\delta$. Define $T^-_\delta f=f_\delta^{-\nu}+ f_\delta^{-\tau}$. Since
	$$ (\dbar\rho^-_\delta)^*\lrcorner T^-_\delta f=(\dbar\rho)^*\lrcorner f^\tau+(\dbar\rho)^*\lrcorner f_\delta^{-\nu}=\psi^0(z)f_N(z)+\sum_{l=1}^{m}\psi^l(z)f_N(z+2\delta\vv{n}^l)=0
	$$
	on $\partial\Omega^-_\delta$, we have $T^-_{\delta} f\in\Dom(Q_{\omz^-_\delta})$. Furthermore,  
\[
\|T^-_{\delta} f-f\|_{\omz^-_\delta}+\|\dbar T^-_{\delta} f-\dbar f\|_{\omz^-_\delta}+\|\vartheta T^-_{\delta} f-\vartheta f\|_{\omz^-_\delta}\to 0
\]
as $\delta\to 0^+$.  The proof for the case  $\Omega^-_\delta$ then follows along the same lines as in the proof of Theorem~\ref{prop5}. For $\Omega^+_\delta$, we set
  	\begin{align}
  		&f_{\delta}^{+\nu}(z)=f^\nu(z)\notag\\
\intertext{and}
  			&f_{\delta}^{+\tau} (z)=\psi^0(z)f^\tau(z)+\sum_{l=1}^{m}\psi^l(z)f^\tau (z-2\delta\vv{n}^l) \notag \\
  	\end{align}
and define $T^+_\delta f=f_{\delta}^{+\tau}+f_\delta^{+\nu}$, and then proceed similarly.
\end{proof}

\section{Lower semi-continuity and  property ($P$)}\label{sec:lower-semi}

Property ($P$) was introduced by Catlin as a potential theoretic sufficient condition for compactness of the inverse of the $\dbar$-Neumann Laplacian on bounded pseudoconvex domains in $\C^n$.  A compact set $K\subset\C^n$ is said to satisfy Property ($P$) if for any $M>0$, there exists a neighborhood $U$ of $K$ and a function $\varphi\in C^\infty(U)$ such that $0\le \varphi \le 1$ and any eigenvalue of the hermitian matrix $(\partial^2\varphi/\partial z_j\partial\bar{z}_k)_{j, k=1}^n$ is greater than or equal to $M$ on $U$. It is said to satisfy Property ($P_q$), $1\le q\le n$, if any sum of $q$ eigenvalues 
of the hermitian metric is greater than or equal to $M$ on $U$. We start with the following well-known lemma.

\begin{lem}\label{P}
	Let $\omz$ be a bounded pseudoconvex domain in $\C^n$ and $b\in C^2(\ol\omz)$ with $-1\le b\le 0$. Then 
	\begin{equation}\label{cat}
	Q_\omz(f,f)\ge \frac{1}{e}\int_\omz H_q(b)(f) dV
	\end{equation}
	for all $f\in\Dom(Q_{q, \omz})$, where $$
	H_q(b)(f)=\mathop{\sumprime}\limits_{|K|=q-1}\sum\limits_{j,k=1}^n\dfrac{\partial^2 b}{\partial z_j\partial \bar z_k}f_{jK}\bar f_{kK}
	$$ 
	is the complex Hessian of $b$, acting on $(0, q)$-forms $f=\sumprime_{|J|=q} f_J\, d\bar z_J$.
\end{lem}
\begin{proof}
	When $\partial\omz$ is smooth, the above  lemma is essentially due to Catlin (see (2.3) in \cite{Catlin84}; see also (2-10) in \cite{BoasStraube99}). When no boundary smoothness is assumed,  the lemma was proved in \cite{Straube97} (see also \cite[Corollary~2.13]{Straube10}). It can also be proved by exhausting $\omz$ from inside by pseudoconvex domains with smooth boundaries and applying Lemma~\ref{lm:T}.
\end{proof}

\begin{lem}\label{lem18}
	Let $\omz$ be a bounded pseudoconvex domain in $\C^n$. Suppose that $\partial\omz$ satisfies property $(P_q)$. Then for any $\eps>0$, there exists a $\delta>0$, such that for any pseudoconvex domain $\omz_j$ with $d_H(\omz, \omz_j)<\delta$, we have
	\begin{align}\label{41}
	\|f_j\|^2_{A_{j\delta}}\leq \eps^2 Q_{\omz_j}(f_j,f_j)
	\end{align}
	for all $f_j\in\Dom(Q_{q, \omz_j})$, where $A_{j\sigma}=\{z\in\omz_j|\dist(z,\partial\omz_j)<\sigma\}$.
	Furthermore,  if $f_j$ is an eigenform of $\Box_{\omz_j}$ with associated eigenvalue $\lambda(\omz_j)$, then 
	\begin{equation}\label{41a}
	\|\dbar_j f_j\|^2_{A_{j\delta}}\leq \eps^2\lambda^2(\omz_j)\|f_j\|_{\omz_j}^2.
	\end{equation} 
	Moreover, if $\partial\Omega$ satisfies property $(P_{q-1})$, then 
	\begin{equation}\label{41b}
	\|\dbarstar_j f_j\|^2_{A_{j\delta}}\leq \eps^2\lambda^2(\omz_j)\|f_j\|_{\omz_j}^2.
	\end{equation}	
\end{lem}
\begin{proof}
	For any $\eps>0$, since $\partial\omz$ satisfies property $(P_q)$, there exists a neighborhood $U$ of $\partial\omz$ and $b\in C^\infty(U)$ with $-1<b\le0$ such that
	\begin{align}
	\sumprime_{K}\sum_{j,k=1}^{n}\dfrac{\partial^2 b}{\partial z_j\partial\bar{z}_k}f_{jK}\bar{f}_{kK}\ge \dfrac{e}{\eps^2}|f|^2
	\end{align}
	for any $(0,q)$-form $f$ on $U$. Let $\delta=\frac{1}{2} \dist(\partial\omz,\partial U)$. If $d_H(\Omega, \Omega_j)<\delta$, then $\partial\omz_j\subset U$ and $A_{j\delta}\subset U\cap\Omega_j$. Applying Lemma \ref{P} to $\omz_j$ and $f_j$, we have
	\begin{equation}\label{j}
	\int_{A_{j\delta}}|f_j|^2dV\le\int_{\omz_j\cap U}|f_j|^2dV\leq \frac{\eps^2}{e}\int_{\omz_j\cap U}H_q(b)(f_j)dV\le \eps^2\,Q_{\omz_j}(f_j,f_j).
	\end{equation}
	This concludes the proof of \eqref{41}. 
	
	To prove \eqref{41a} and \eqref{41b}, we first note that if $f_j$ is an eigenform for $\Box_{\Omega_j}$ associated with eigenvalue $\lambda(\Omega_j)$, then 
	\begin{align}
	\dbarstar_j\dbar_j f_j&=\lambda(\omz_j) f_j-\dbar_j\dbarstar_j f_j\in\Dom(\dbar_{j})\\
	\intertext{and}
	 \dbar_j\dbarstar_j f_j&=\lambda(\omz_j) f_j-\dbarstar_j\dbar_j f_j\in\Dom(\dbarstar_j).
	\end{align} 
   Moreover, 
	\begin{align}
	\Box_{\omz_j}\dbar_jf_j&=\dbar_j\Box_{\Omega_j} f_j=\lambda(\omz_j)\dbar_j f_j\\
	\intertext{and} 
	\Box_{\omz_j}\dbarstar_jf_j&=\dbarstar_j\Box_{\Omega_j}f_j=\lambda(\omz_j)\dbarstar_j f_j.
	\end{align} 
	Since $\partial\Omega$ satisfies property ($P_q$), it also satisfies property ($P_{q+1}$). Therefore, applying \eqref{41} to $\dbar_j f_j$, we have
	\begin{align*}
	\|\dbar_j f_j\|^2_{A_{j\delta}}&\leq \eps^2Q_{\omz_j}(\dbar_j f_j,\dbar_j f_j) =\eps^2\lambda(\omz_j)\,\|\dbar_j f_j\|^2_{\omz_j}\\
	&\leq \eps^2\lambda(\omz_j)\,Q_{\omz_j}(f_j,f_j)=\eps^2\lambda^2(\omz_j)\|f_j\|_{\omz_j}^2.
	\end{align*}
	Similarly, when $\partial\Omega$ satisfies property ($P_{q-1}$), we have
	$$\|\dbarstar_jf_j\|^2_{A_{j\delta}}\le \eps^2Q_{\omz_j}(\dbarstar_j f_j,\dbarstar_j f_j)\le \eps^2\lambda^2(\omz_j)\|f_j\|_{\omz_j}^2.$$ 
	This concludes the proof of Lemma~\ref{lem18}.
\end{proof}

\begin{lem}\label{Pj}
	$\omz$ is a bounded pseudoconvex domain in $\C^n$ that satisfies property $(P_q)$. Let $M$ be a positive constant. Let $\{\omz_j\}$ be a family of pseudoconvex domains such that $d_H(\omz,\omz_j)\to0$ as $j\to\infty$. Suppose $f_j\in\Dom(Q_{\omz_j})$ is a sequence of $(0,q)$-forms such that $\|f_j\|^2_{\omz_j}+Q_{\omz_j}(f_j,f_j)\le M$. Let $\tf_j$ be the extension of $f_j$ to 0 outside of $\omz_j$. Then $\{\tf_j\}$ is a pre-compact family in $L^2_{(0,q)}(\C^n)$.
\end{lem}
\begin{proof}
	For any $\eps>0$, it follows from (\ref{j}) that there exist a neighborhood $U$ of $ \partial\omz$ such that 
	$$\int_{U}|\tf_j|^2dV=\int_{\omz_j\cap U}|f_j|^2dV\leq  \eps^2\,Q_{\omz_j}(f_j,f_j)\le \eps^2 M$$
	for sufficiently large $j$.
	Let $V\subset\subset U$ be a neighbourhood of $\partial\omz$. Choosing sufficiently large $j$ such that $(\omz\setminus\omz_j)\cup(\omz_j\setminus\omz)\subset V$. Let $\eta\in C_0^{\infty}(\C^n)$ with $0\le\eta\le1$, $\eta\equiv1$ on $\omz\setminus U$ and $\supp\eta\subset\omz\setminus V$. Then there exists a constant $M_1$, such that $Q_{\omz_j}(\eta f_j,\eta f_j)\le M_1$ and hence $\|\eta \tf_j\|_{W^1(\omz)}\le M_1$.  (Hereafter $\|f\|_{W^\alpha}$ denotes the norm of $L^2$-Sobolev space of order $\alpha$.)  By Rellich's compactness theorem, $\{\eta\tf_j\}$ has a subsequence $\{\eta\tf_{j_l}\}$ that conveges in $L^2_{(0,q)}(\omz)$. Thus
	\begin{align*}
    \|\tf_{j_h}-\tf_{j_l}\|_{\C^n}&=\|\tf_{j_h}-\tf_{j_l}\|_{U}+\|\tf_{j_h}-\tf_{j_l}\|_{\omz\setminus U}\\
    &\le 2M^{1/2}\eps+\|\eta\tf_{j_h}-\eta\tf_{j_l}\|_{\omz},
	\end{align*}
	when $h$ and $l$ are sufficiently large.
	Thus $\{\tf_{j_l}\}$ is a subsequence of $\{\tf_j\}$ that converges in $L^2_{(0,q)}(\C^n)$.
\end{proof}

\begin{remark}\label{remark2}
	Let  $f_j$ be an eigenform associated with $k^{\rm th}$ eigenvalue $\lambda_k(\Omega_j)$ of $\Box_{\Omega_j}$.  From the proof of Lemma \ref{lem18}, we know that $\dbar_jf_j$ is also an eigenform of $\Box_{\omz_j}$.  Moreover, 
	$$
	\|\dbar_j f_j\|^2_{\omz_j}+Q_{\omz_j}(\dbar_jf_j,\dbar_jf_j)\le \lambda_k(\omz_j)(1+\lambda_k(\omz_j)) \|f_j\|^2_{\Omega_j},
	$$
	which, by Corollary~\ref{cor1}, is bounded from above by a constant independent of $j$. Therefore, $\{\widetilde{\dbar_j f_j}\}$ is also a pre-compact family in $L^2_{(0,q+1)}(\C^n)$.  Similarly, when $\partial\Omega$ satisfies property $(P_{q-1})$,	$\{\va \tf_j\}$ is also a pre-compact family in $L^2_{(0, q-1)}(\C^n)$.
\end{remark}

\begin{thm}\label{cor3}
	Let	$\omz$ be a bounded pseudoconvex domains in $\C^n$ with $C^1$ boundary that satisfies Property $(P_{q-1})$, $2\le q\le n-1$. Let $\omz_j$ be a sequence of bounded pseudoconvex domains whose $\dbar$-Neumann Laplacian $\Box_{\Omega_j}$ has purely discrete spectrum on $(0, q)$-forms. If $d_H(\omz_j, \omz)\to 0$ as $j\to\infty$, then for any $k\in\nb$, 
	\begin{equation}\label{PP}
	\liminf_{j\to\infty}\lambda^q_k(\Omega_j)\ge \lambda^q_k(\Omega).
	\end{equation}
\end{thm}

\begin{proof}
	 The proof is similar in some respects to Theorem \ref{prop5}. The difference here is to use Lemma~\ref{Pj} and the Kolmogorov-Riesz theorem to establish estimates that are uniform with regard to $j$.
	 
	 We first construct the transition operator $T_{j\delta}$ from $\Dom(Q_{\Omega_j})$ into $\Dom(Q_\Omega)$. Let $\{U^l\}_{l=0}^m$ be an open covering of $\ol{\Omega}$ and let $\{\psi^l\}_{l=0}^m$ be a partition of unity subordinated to this covering,  constructed as in the setup preceding Lemma~\ref{lm:hatcheck}.  Let $U=\cup_{l=1}^m U^l$ and let $V\subset\subset U$ be a tubular neighborhood of $\partial\Omega$ such that $\dist(\partial V, \partial\Omega)<\dist(\partial U, \partial\Omega)$. We assume that $j$ is sufficiently large so that $(\Omega\setminus\Omega_j)\cup (\Omega_j\setminus\Omega)\subset\subset V$.  Let $f_j\in\Dom(Q_{\Omega_j})$. For any $\delta<\dist(\partial V, \partial\Omega)$ and any sufficiently large $j$ such that  $\delta_j=d_H(\Omega_j, \Omega)<\delta$, we define 
	 \begin{equation}
	 	\hf_{j\delta}(z)=\psi^0(z)f_j(z)+\sum_{l=1}^{m}\psi^l(z)f_j(z-2\delta\vv{n}^l)	
	 \end{equation}
	 and
	 \begin{equation}
	 	\cf_{j\delta}(z)=\psi^0(z)\tf_j(z)+\sum_{l=1}^{m}\psi^l(z)\tf_j(z+2\delta\vv{n}^l).	
	 \end{equation} 
	 (Throughout this proof, we will use $\widetilde f_j$ to denote the form obtained by extending $f_j$ to $0$ outside of $\Omega_j$.) Notice that $z-2\delta \vv n^l\in\Omega_j$ and $z+2\delta\vv n^l\not\in\Omega_j$ for $z\in\Omega\cap U^l$ (see the proof of Lemma~\ref{lm:hatcheck}). It follows that $\hf_{j\delta}\in\Dom(\dbar_\omz)$ and $\cf_{j\delta}\in\Dom(\dbarstar_\omz)$.   Define
	 	$T_{j\delta}\colon \Dom(Q_{\Omega_j})\to \Dom(Q_\Omega)$  by
	 	\begin{align}
	 		T_{j\delta} f_j=\dbarstar N\dbar \widehat f_{j\delta}+\dbar N\dbarstar\cf_{j\delta}.
	 	\end{align}
	 Then
	 	\begin{equation}
	 			T_{j\delta}f_j=\dbarstar\dbar N \hf_{j\delta}+\dbar\dbarstar N\cf_{j\delta}
	 	\end{equation}		
	 	and		
	 	\begin{equation}
	 		\dbar T_{j\delta} f_j=\dbar\dbarstar\dbar N\hf_{j\delta}=\dbar\hf_{j\delta} \quad\text{and}\quad 
	 		\dbarstar T_{j\delta} f_j=\dbarstar\dbar\dbarstar N \cf_{j\delta}=\dbarstar \cf_{j\delta}.
	 	\end{equation}
	 
	 We first fix $1\le l\le k$ and let $f_{j}$ be the normalized eigenform of $\Box_{\Omega_j}$ associated eigenvalue $\lambda_l(\omz_j)$. Since
	 \[
	 T_{j\delta} f_j-\tf_j=\dbarstar\dbar N (\hf_{j\delta}-\tf_{j})+\dbar \dbarstar N(\cf_{j\delta}-\tf_j),
	 \]
	 we have
	 \[
	 \begin{aligned}
	 \|T_{j\delta} f_j -\tf_j\|_\Omega &\le \|\hf_{j\delta}-\tf_{j}\|_\Omega+\|\cf_{j\delta}-\tf_j\|_\Omega\\
	 &\le	\sum_{l=1}^m \left(\left\|\tf_j(z-2\delta\vv{n}^l)-\tf_j(z)\right\|_{\omz\cap U^l}
	 + \left\|\tf_j(z+2\delta\vv{n}^l)-\tf_j(z)\right\|_{\omz\cap U^l}\right) \\
	 &\le	\sum_{l=1}^m \left(\left\|\tf_j(z-2\delta\vv{n}^l)-\tf_j(z)\right\|_{\C^n}
	 + \left\|\tf_j(z+2\delta\vv{n}^l)-\tf_j(z)\right\|_{\C^n}\right).
	 \end{aligned}
	 \]
	 By Lemma~\ref{Pj} and the subsequent remark, $\{\tf_j\}$ is a pre-compact family in $L^2_{(0, q)}(\C^n)$. For any $0<\eps<1$, it the follows from the Kolmogorov-Riesz theorm that 
	 \[
	  \|T_{j\delta} f_j -\tf_j\|_\Omega<\eps
	 \]
	 for all sufficiently small $\delta$ and sufficiently large $j$.
	 
	 Furthermore, we have
	 \[
	 \begin{aligned}
	 &\|\dbar T_{j\delta} f_j -\widetilde{\dbar f_j}\|_\Omega =\|\dbar \widehat{f}_{j\delta} -\widetilde{\dbar f_j}\|_\Omega\\
	 &\quad\le C\sum_{l=1}^m \big(\big\|\tf_j(z-2\delta\vv{n}^l)-\tf_j(z)\big\|_{\omz\cap U^l}
	 + \big\|\dbar f_j(z-2\delta\vv{n}^l)-\widetilde{\dbar f_j}(z)\big\|_{\omz\cap U^l}\big) \\
	 &\quad\le C\sum_{l=1}^m \big(\big\|\tf_j(z-2\delta\vv{n}^l)-\tf_j(z)\big\|_{\C^n}
	 + \big\|\widetilde{\dbar f_j}(z-2\delta\vv{n}^l)-\widetilde{\dbar f_j}(z)\big\|_{\C^n}\big), 
	 \end{aligned}
	 \]
	 where the constant $C$ depends only on the partition of unity and is independent of $\delta$ or $j$. Note that $\widetilde{\dbarstar f_j}=\dbarstar\widetilde{f_j}$.  Using the pre-compactness of the families $\{\tf_j\}$ and $\{\widetilde{\dbar f_j}\}$ in $L^2$-spaces, we then have
	 \[
	 \|\dbar T_{j\delta} f_j -\widetilde{\dbar f_j}\|_\Omega<\eps
	 \]
	 for all sufficiently small $\delta$ and all sufficiently large $j$. Similarly, 
	 \[
	 \|\dbarstar T_{j\delta} f_j -\widetilde{\dbarstar f_j}\|_\Omega 
	 =\|\dbarstar \widecheck{f}_{j\delta} -\widetilde{\dbarstar f_j}\|_\Omega<\eps.
	 \]

	 The rest of the proof follows the same lines of arguments as in the proof of Theorem~\ref{prop5}. We sketch the proof below. Let $f_{jl}$ be the normalized eigenform of $\Box_{\Omega_j}$ associated with eigenvalue $\lambda_l(\omz_j)$. Let $L_{jk}$ be the $k$-dimensional linear subspace of $\Dom(Q_{\Omega_j})$ spanned by $\{f_{jl}\}_{l=1}^k$.  Let $f_j=\sum_{l=1}^k c_l f_{jl}$ be a $(0, q)$-form in $L_{jk}$.  Then
\begin{equation}\label{tp}
	\begin{aligned}
		\big|\|T_{j\delta} f_j\|^2_{\Omega} - \|f_j\|^2_{\Omega_j} \big|&=\big|\| T_{j\delta} f_j\|^2_{\Omega}-\|\widetilde{f_j}\|^2_{\Omega}-\|f_j\|^2_{\Omega_j\setminus\Omega}\big| \\
		&\le \big(\|T_{j\delta} f_j\|_{\Omega}+\|\widetilde{f_j}\|_{\Omega}\big)\|T_{j\delta} f_j-\widetilde{f_j}\|_{\Omega}+\|f_j\|^2_{\Omega_j\setminus\Omega}\\
		&\le C\eps \|f_j\|^2_{\Omega_j}
	\end{aligned}
\end{equation}	
	for all sufficiently small $\delta$ and sufficiently large $j$. Note that in the last inequality, we have used Lemma~\ref{lem18}.  Similarly, we have
	\[
	\big| Q_{\Omega}(T_{j\delta} f_j, T_{j\delta} f_j)-Q_{\Omega_j}(f_j, f_j)\big|\le C\eps \|f_j\|^2_{\Omega_j}.
	\]
	 The desired inequality \eqref{PP} then follows from Lemma~\ref{first} and the subsequent remark. \end{proof}	

Theorem~\ref{main2} is a direct consequence of Theorem~\ref{cor3} by {\it reductio ad absurdum}. Combining Theorem~\ref{prop1} and Theorem~\ref{main2}, we then have:

\begin{cor} Let	$\omz, \Omega_j$ be bounded pseudoconvex domains in $\C^n$. Suppose $\partial\Omega$ is $C^1$-smooth and satisfies Property $(P_{q-1})$, $2\le q\le n-1$ and $\partial\omz_j$ satisfies property $(P_q)$. If $d_H(\omz_j, \omz)\to 0$ as $j\to\infty$, then for any $k\in\nb$, 
	\begin{equation}
	\lim_{j\to\infty}\lambda^q_k(\Omega_j)= \lambda^q_k(\Omega).
	\end{equation}	
\end{cor}
 
\begin{remark} Unlike Theorem~\ref{prop3}, lower semicontinuity property does not hold on level sets of a smooth bounded domain without additional assumption. For example, let $\Omega$ be the Diederich-Fornaess worm domain with winding greater than $\pi$. Since $\Omega$ does not have Stein neighborhood basis $($\cite{DiederichFornaess77b}$)$, 
we have
This follows from the fact that pseudoconvexity of a smooth bounded domain in $\C^2$ is characterized by positivity of one of the variational eigenvalues $\lambda^1_k(\Omega)$ $($see \cite{Fu10}$)$. 
\end{remark}

\section{Quantitative estimates  on finite type domains}\label{sec:finite-type}

We continue our study of spectral stability on smooth bounded pseudoconvex domains of finite type. Our aim is to establish quantitative estimates for the stability on such domains. 
Notions of finite type were introduced by Kohn \cite{Kohn72}, D'Angelo \cite{Dangelo82, Dangelo93}, and Catlin \cite{Catlin83, Catlin84b, Catlin87} in connection with subelliptic theory of the $\dbar$-Neumann Laplacian. A smooth bounded domain $\Omega$ in $\C^n$ is said to be of finite $D_q$-type if the order of contact of $\partial\Omega$ with any $q$-dimensional complex analytic variety is finite. (We refer the reader to \cite{Dangelo82, Dangelo93} for precise definitions.) 

A fundamental theorem of Catlin states that a smooth bounded pseudoconvex domain $\Omega$ in $\C^n$ is of finite $D_q$-type if and only if the $\dbar$-Neumann Laplacian satisfies the following subelliptic estimate
\begin{equation}\label{subelliptic}
\|f\|^2_{W^\alpha}\le C Q_\Omega(f, f), \quad \forall f\in\Dom(Q_{q, \Omega})
\end{equation}
for some constants $0<\alpha\le 1/2$ and $C>0$. The constant $\alpha$ is referred to as the order of subellipticity. A key step in Catlin's theory is the construction of plurisubharmonic functions with large complex Hessians. More precisely, if $\Omega$ is a smooth bounded pseudoconvex domain in $\C^n$ of finite type, then there exist constants $\alpha>0$, $\delta_0>0$, and $C>0$ such that for any $0<\delta<\delta_0$, there exists a smooth plurisubharmonic function $\lambda_\delta$ on $\overline{\Omega}$ with $|\lambda_\delta|\le 1$ and
\begin{equation}\label{catlin1}
H_q(\lambda_\delta)(f)\ge C |f|^2/\delta^{2\alpha}
\end{equation}
on $A_\delta=\{z\in\Omega \mid d(z)=\dist(z, \partial\Omega)<\delta\}$ (\cite[Theorem~9.2]{Catlin87}). Subelliptic estimate \eqref{subelliptic} is then a consequence of the existence of such plurisubharmonic functions. Straube \cite{Straube97} showed that this last step also holds on bounded pseudoconvex domains with Lipschitz boundaries: Let $\Omega$ be a bounded pseudoconvex domain in $\C^n$ with Lipschitz boundary. Suppose there exist a continuous plurisubharmonic function $\lambda$ on $\Omega$ and constants $\alpha>0$, $C>0$ such that
\begin{equation}\label{catlin2}
H_q(\lambda)(f)\ge C |f|^2/(d(z))^{2\alpha}
\end{equation}
on $\Omega$ as currents, then subelliptic estimate \eqref{subelliptic} holds. For abbreviation, a bounded pseudoconvex domain $\Omega$ is said to satisfy {\it property} ($P_q^\alpha$) if condition \eqref{catlin1} is satisfied. We have the following simple analogues of Lemma~\ref{lem18}.

\begin{lem}\label{lem13}
	Let $\omz$ be a bounded pseudoconvex domain in $\C^n$. Suppose $\Omega$ satisfies property $(P_q^\alpha)$. Then there exists a  constant $C$ such that for all sufficiently small $\delta>0$,
	\begin{align}\label{22}
	\|f\|^2_{A_\delta}\leq C\delta^{2\alpha} Q_\omz(f,f),\qquad\forall f\in \Dom(Q_{q, \omz}).
	\end{align}
	Furthermore,  if $f$ is an eigenform for $\Box_\Omega$ associated with eigenvalue $\lambda(\Omega)$, then 
	\begin{align}\label{a1}
	\|f\|^2_{A_\delta}\leq C\delta^{2\alpha}\lambda(\omz)\|f\|^2_\omz \quad\text{and}\quad \|\dbar f\|^2_{A_{\delta}}\leq C\delta^{2\alpha}\lambda^2(\omz)\|f\|_{\omz}^2.
	\end{align} 
	Moreover, if $\Omega$ satisfies property $(P_{q-1}^\alpha)$, then 
	\begin{align}\label{a2}
	\|\dbarstar f\|^2_{A_{\delta}}\leq C\delta^{2\alpha}\lambda^2(\omz)\|f\|_{\omz}^2.
    \end{align}
\end{lem}

\begin{lem}\label{lem16}
	Let $\omz$ be a bounded pseudoconvex domain in $\C^n$. Suppose $\Omega$ satisfies property $(P_q^\alpha)$. Then there exists a constant $\delta>0$ such that for any pseudoconvex domain $\omz_j$ with $\delta_j=d_H(\omz, \omz_j)\le\delta$, we have
	\begin{align}\label{39}
	\|f_j\|^2_{A_{j\delta_j}}\leq C\delta_j^{2\alpha} Q_{\omz_j}(f_j,f_j),\qquad\forall f_j\in\Dom(Q_{\omz_j}),
	\end{align}
	where $A_{j\sigma}:=\{z\in\omz_j|\dist(z,\partial\omz_j)<\sigma\}$.
	Furthermore,  if $f_j\in \Dom(\Box_{\omz_j}) $ is an eigenform satisfies $\Box_{\omz_j} f_j=\lambda(\omz_j) f_j$, then 
	\begin{align}
	\|\dbar_j f_j\|^2_{A_{j\delta_j}}\leq C\delta_j^{2\alpha}\lambda^2(\omz_j)\|f_j\|_{\omz_j}^2.
	\end{align} 
	Moreover, if $\Omega$ satisfies property $(P_{q-1}^\alpha)$, then
	 \begin{equation}
	 \|\dbarstar_j f_j\|^2_{A_{j\delta_j}}\leq C\delta_j^{2\alpha}\lambda^2(\omz_j)\|f_j\|_{\omz_j}^2.
	 \end{equation}
\end{lem}

These two lemmas are simple consequence of Lemma \ref{P}, following the same line
of arguments as in Lemma~\ref{lem18}. We omit the proofs. The following lemma is a direct consequence of the interior ellipticity of the $\dbar\oplus\dbarstar$. 

\begin{lem}\label{Interior} Let
	$\omz$ be a bounded domain in $\C^n$. Let $\omz_{\delta}=\{z\in\omz\,|\,\dist(z,\partial\omz)>\delta\}$. Then there exists a constant $C>0$ such that 
	\begin{align}\label{10}
	\|f\|^2_{W^1(\omz_\delta)}\le C\big(Q_\Omega(f, f)+ \frac{1}{\delta^2}\|f\|^2_\Omega\big)
	\end{align}
	for all $f\in\Dom(Q_\Omega)$.
\end{lem}

\begin{proof} The lemma is also well know. We include a proof for the reader's convenience. Let $\chi(t)=0$ for $t<1/2$, $\chi(t)=2(t-1/2)$ for $t\in [1/2, 1]$, and $\chi(t)=1$ for $t>1$. Let $d(z)=\dist(z,\partial\omz)$ and $\eta(z)=\chi(d(z)/\delta)$. Note that since the distance function is uniformly Lipschitz with Lipschitz constant $1$,  we have $|\nabla\eta(z)|\le 2/\delta$ almost everywhere on $\Omega$. Therefore
\[
\sum_{l=1}^n\sumprime_{|J|=q} \big\|\partial (\eta f_J)/\partial\bar z_l\big\|^2
\le Q(\eta f, \eta f)\le C\big(Q_\Omega(f, f)+\frac{1}{\delta^2}\|f\|^2_\Omega\big).
\]
(See \cite[Corollary~2.13]{Straube10} for a proof of the first inequality.) The desired inequality then follows from integration by part on the left-hand side.
\end{proof}

We remark that the constant in \eqref{10} can be chosen to be independent of $\Omega$. We will use this fact in the proof of the next theorem.

\begin{thm}\label{thm6} 
Let $\omz$ be a bounded pseudoconvex domain with $C^1$-smooth boundary. Assume that $\Omega$ satisfies property $(P_{q-1}^\alpha)$. Let $\omz_j$ be a bounded pseudoconvex domain whose $\dbar$-Neumann Laplacian has discrete spectrum on $(0, q)$-forms. Let $k\in\nb$. Then there exist constants $\delta>0$ and $C>0$ such that 
	\begin{equation}\label{44}
	\left|\lambda_k(\omz_j)-\lambda_k(\omz)\right|\le Ck\delta_j^{\alpha/(\alpha+1)} (\lambda_k(\omz)+1)^2,
	\end{equation}
provided $\delta_j=d_H(\omz,\omz_j)<\delta$.
\end{thm}

\begin{proof}	The proof follows the same line of arguments as those for Theorem~\ref{cor3}. The difference here is that we use Lemma~\ref{lem13} and Lemma~\ref{lem16} to estimate terms near the boundary and use Lemma~\ref{Interior} to estimate terms inside the domain. 

We provide the proof of the inequality 
 $$
 \lambda_k(\Omega)-\lambda_k(\Omega_j)\le Ck\delta_j^{\alpha/(\alpha+1)}.
 $$
Following the same setup as in the proof of Theorem~\ref{cor3}, for $f_j\in\Dom(Q_{\Omega_j})$, we set
\begin{equation}\label{1-est}
\hf_{j}(z)=\psi^0(z)f_j(z)+\sum_{l=1}^{m}\psi^l(z)f_j(z-2\delta_j\vv{n}^l)	
\end{equation}
and
\begin{equation}
\cf_{j}(z)=\psi^0(z)\tf_j(z)+\sum_{l=1}^{m}\psi^l(z)\tf_j(z+2\delta_j\vv{n}^l).	
\end{equation} 
Define
$T_{j}\colon \Dom(Q_{\Omega_j})\to \Dom(Q_\Omega)$  by
\begin{align}
T_{j} f_j=\dbarstar N\dbar \widehat f_{j}+\dbar N\dbarstar\cf_{j}.
\end{align}

We now assume that $f_j(z)$ is the normalized eigenform of $\Box_{\Omega_j}$ associated with the eigenvalue $\lambda(\Omega_j)$. As in the proofs of Theorems~\ref{prop5} and~\ref{cor3}, it suffices to estimate the terms
\begin{equation}\label{term1}
\|f_j\|_{\omz_j\setminus\omz}, \quad \|\dbar f_j\|_{\omz_j\setminus\omz}, \quad \|\dbarstar f_j\|_{\omz_j\setminus\omz}, \quad \|\hf_j-\tf_j\|_{\omz},\quad  \|\cf_j-\tf_j\|_{\omz},
\end{equation}
and
\begin{equation}\label{term2}
\|\dbar \hf_j-\widetilde{\dbar f_j}\|_{\omz}, \quad \|\dbarstar\cf_j-\widetilde{\dbarstar f_j}\|_{\omz}.
\end{equation}
From \eqref{39} in Lemma~\ref{lem16}, we have  
$$
\|f_j\|_{\omz_j\setminus\omz}\le \|f_j\|_{A_{j\delta_j}}\le C\delta_j^{\alpha}(\lambda(\omz_j))^{1/2}.
$$
Similarly,
$$
\|\dbar f_j\|_{\omz_j\setminus\omz}+\|\dbarstar f_j\|_{\omz_j\setminus\omz}\le \|\dbar f_j\|_{A_{j\delta_j}}+\|\dbarstar f_j\|_{A_{j\delta_j}}\le C\delta_j^{\alpha}\lambda(\omz_j).
$$
Noticing that from Lemma~\ref{Interior}, we have
\[
\|\nabla f_j \|_{\Omega_{j, \delta_j^\beta}}\le C\delta_j^{-\beta}(\lambda(\Omega_j)+1)^{1/2},
\]
where $\Omega_{j, \delta_j^\beta}=\{z\in\Omega_j  \mid  \dist(z, \partial\Omega_j)>\delta_j^\beta\}$. Taking $\beta=1/(1+\alpha)$, we then have
     \begin{equation}
	\begin{aligned}
	\|\hf_j -\tf_j\|_{\omz}&\le C\sum_{l=1}^{m} \left\|f_j(z-2\delta_j\vv{n}^l)-\tf_j(z)\right\|_{\omz\cap U^l}\\
	&\le C\sum_{l=1}^m\left\|f_j(z-2\delta_j\vv{n}^l)-f_j(z)\right\|_{(\omz\cap U^l)\setminus A_{j, \delta_j^{\beta}}}+C\|f_j(z)\|_{A_{j, 2\delta_j^{\beta}}}\\
	&\le C\delta_j\left\|\nabla f_j\right\|_{\Omega_{j, \frac{1}{2}\delta_j^{\beta}}}+C\delta_j^{\alpha/(\alpha+1)}(\lambda(\omz_j))^{1/2} \\
	&\le C\delta_j^{\alpha/(\alpha+1)}(\lambda(\omz_j)+1)^{1/2}.
	\end{aligned}
	\end{equation}
Noticing that in obtaining the last inequality, we have used the facts that 
$$
\Omega\setminus A_{j, \delta^\beta_j}\subset \Omega_{j, \frac{1}{2}\delta^\beta_j}
$$ 
and $\lambda(\Omega_j)$ is controlled from above by the corresponding eigenvalue $\lambda(\Omega)$ of $\Box_\Omega$ as shown in Corollary~\ref{cor2}. Similar estimates also hold
for the other three terms in \eqref{term1} and \eqref{term2}. Notice that plugging $\dbar f_j$ and $\dbarstar f_j$ into \eqref{10}, we have
\[
\|\nabla \, \dbar f_j\|_{\Omega_{j, \delta}} \le C \delta^{-1}(\lambda(\Omega_j)+1) \quad \text{and}\quad \|\nabla \, \dbarstar f_j\|_{\Omega_{j, \delta}} \le C \delta^{-1}(\lambda(\Omega_j)+1)
\]
and the constants in the above estimates are independent of $j$. Using Lemma \ref{first}, we then obtain inequality \eqref{1-est}. The proof of the other inequality in Theorem~\ref{thm6} is similar and is left to the interested reader. 
\end{proof}

Quantitative estimate \eqref{44} can be sharpened when more restriction is placed on the boundaries of $\Omega$ and $\Omega_j$. A family of smoothly bounded pseudoconvex domains $\Omega_j$ in $\C^n$ with defining functions $\rho_j$ is said to be of {\it uniform finite $D_q$-type} if there exist positive constants $\alpha$ and $C$ such that inequality \eqref{catlin1} holds for all $\Omega_j$ and the $C^\infty$-norm of  $\rho_j$ is uniformly bounded. The following lemma is a direct consequence of Catlin's subelliptic estimates (\cite{Catlin87}).  

\begin{lem}\label{lem15} Let $\omz$ be a smooth bounded pseudoconvex domain of finite type in $\C^n$. Let $l$ be a non-negative integer. Let $f$ be an eigenform of the $\dbar$-Neumann Laplacian $\Box_{q, \omz}$ with associated eigenvalue $\lambda(\Omega)$. Then there exist positive constants $\alpha$ and $C_l$ such that 
	\begin{align}\label{28}
	\|f\|_{C^l(\overline{\Omega})}\le C_l (\lambda(\omz))^{[\frac{n+l}{2\alpha}]+1} \|f\|,
	\end{align}
	where $[(n+l)/2\alpha]$ denotes the integer part of $(n+l)/2\alpha$.
\end{lem}
\begin{proof}  It follows from above-mentioned work of Catlin that $\Omega$ satisfies property $(P_q^\alpha)$ for some $\alpha\in (0, \ 1/2]$ and there exists a constant $C_s>0$ such that
\begin{equation}\label{cat-est}
\|N_\omz f\|_{W^{s+2\alpha}(\omz)}\le C_s\|f\|_{W^s(\omz)}.
\end{equation}
Starting with $s=0$ and repeatedly applying \eqref{cat-est} to $\Box_\Omega f=\lambda(\Omega) f$, we then have
\begin{equation}\label{28b}
	\|f\|_{W^{2m\alpha}(\omz)}\le C(\lambda(\omz))^{m}\|f\|,\quad m\in\N.
\end{equation}
	The desired estimates \eqref{28} is then an immediate consequence of Sobolev embedding theorem.
\end{proof}

We remark that the constant in \eqref{28} depends only on the constant in \eqref{catlin1} and the $C^\infty$-norm of the defining function of $\Omega$. We will use this fact in proving the  following theorem:

\begin{thm}\label{prop10} 
	Let $\omz$ be a smooth bounded pseudoconvex domain of finite $D_q$-type in $\C^n$. Let $\Omega_j$ be a family of bounded pseudoconvex domains. Let $1\le q\le n-1$ and let $k\in\nb$. 
	Then there exist 
	constants $C_k>0$ and $\delta>0$ such that 
	\begin{equation}\label{ft-1}
		\lambda^q_k(\omz_j)-\lambda^q_k(\omz)\le C_k\delta^{1/2}_j,
	\end{equation}
	provided $\delta_j=d_H(\omz,\omz_j)<\delta$. Furthermore, if $\Omega_j$ is a family of smooth bounded pseudoconvex domains of uniform finite $D_q$-type, then 
	\begin{equation}\label{ft-2}
		-C_k\delta_j\le\lambda^q_k(\omz_j)-\lambda^q_k(\omz)\le C_k \delta_j.
	\end{equation}
\end{thm}

We will prove this theorem using the following sharp Hardy's inequality due to Brezis and Marcus (for functions)~\cite{BrezisMarcus97} and an idea from Davies~\cite{Davies00}. 

\begin{lem}\label{Hardy1} Let $\omz$ be a bounded domain in $\C^n$ with $C^2$-boundary. Then there exists a constant $A>0$ such that
 	\begin{equation}\label{8}
 	\int_{\omz}\frac{|f_N|^2}{(d(z))^2}dV\leq 16\left(Q_{\Omega}(f_N,f_N)+A\|f_N\|^2\right)
    \end{equation} 
    for any $f\in \Dom(Q_{\Omega})$, where $f_N=(\dbar d(z))^*\lrcorner f$ is the normal component of $f$ and $d(z)=d(z, \partial\Omega)$ is the Euclidean distance from $z$ to the boundary $\partial\omz$. Furthermore, if $\Omega$ is pseudoconvex, then for any $\eps>0$, there exists a constant $C_\eps>0$ such that
 	\begin{equation}\label{8a}
 	Q_{\Omega}(f_N, f_N)\le (q/4+ \eps)Q_{\Omega}(f, f)+C_\eps \|f\|^2.
 	\end{equation}	
 \end{lem}
 \begin{proof} For functions in $W^1_0(\Omega)$, inequality \eqref{8} was proved in \cite{BrezisMarcus97}. We first provide a proof for functions, using only the divergence theorem. We can assume that $f\in C^\infty_0(\Omega)$.  Replacing $d(z)$ by a function that is identical to $d(z)$ in a neighborhood of $\partial\Omega$ and $C^2$ inside $\Omega$, we may assume that $d\in C^2(\overline{\Omega})$. Then
     \begin{equation}\label{eq:h1}
 	 \begin{aligned}
 	 &\int_{\omz}|\nabla f|^2 dV=\int_{\omz}\left| d^{1/2}\nabla (d^{-1/2}f)+\frac{\nabla d}{2d}f\right|^2 dV\\
 	 &=\int_{\omz} d|\nabla(d^{-1/2}f)|^2 dV+\frac{1}{2}\langle\nabla |d^{-1/2}f|^2, \nabla d\rangle_\omz+\frac{1}{4}\int_{\omz}\frac{|\nabla d|^2}{d^2}|f|^2 dV.\\
 	 \end{aligned}
 	 \end{equation}
 	 Let $g=d^{-1/2}f$. Note that $g\in C^\infty_0(\omz)$. By the divergence theorem, we have
 	 \begin{equation}
 	 \begin{aligned}
 	 0=\int_\omz \nabla\cdot (|g|^2 d\nabla d)\, dV&=\int_\omz \big(\nabla |g|^2\cdot (d\nabla d)+ |g|^2 |\nabla d|^2+ |g|^2 d\nabla^2 d \big)\, dV.
 	 \end{aligned}
 	 \end{equation}
 	 Thus,
 	 \begin{equation}
 	 \begin{aligned}
 	 \int_\omz |g|^2 |\nabla d|^2\, dV&=-2\rea\int_\omz \bar g\nabla g\cdot d(\nabla d)\,dV-\int_\omz |g|^2 d(\nabla^2d)\, dV\\
 	 &\le \eps\int_\omz d|\nabla g|^2dV+\dfrac{1}{\eps}\int_\omz d|\nabla d|^2|g|^2dV+\int_\omz d|\nabla^2d||g|^2dV\\
 	 &\le \eps\int_\omz d|\nabla g|^2dV+C_\eps\int_\omz d|g|^2dV.
 	 \end{aligned}
 	 \end{equation}
 	 Note that $|\nabla d|=1$ near $\partial\omz$. The middle term in the last expression of \eqref{eq:h1} is under controlled as above.  Thus by choosing $A>0$ sufficiently large, we obtain the following version of Hardy's inequality:
 	 \begin{equation}\label{11}
 	 \begin{aligned}
 	 \int_{\omz}\frac{|f|^2}{d^2} dV\le 4\int_{\omz}|\nabla f|^2 dV+A\|f\|^2.
 	 \end{aligned}
 	 \end{equation}
 	 The above inequality holds for all $f\in W^1_0(\Omega)$ as $C^\infty_0(\Omega)$ is dense in $W^1_0(\Omega)$. 
 	 
     Since $C^1_{(0, q)}(\overline{\Omega})\cap\Dom(Q_\Omega)$ is dense in $\Dom(Q_\Omega)$ in graph norm, it suffices to prove \eqref{8} for $f\in C^1_{(0,q)}(\bar\omz)\cap\Dom(Q_\Omega)$. Notice that
     $$
     f_N=(\dbar d(z))^*\lrcorner f= \sumprime_{|K|=q-1}f_N^Kd\bar{z}_K \quad \text{where}\quad 
     f_N^K=\sum_{j=1}^{n}\frac{\partial d(z)}{\partial z_j}f_{jK}.
     $$ 
     Since $f\in\Dom(\dbarstar_\Omega)$, it follows that $f_N^K=0$ on $\partial\omz$ and hence $f_N^K\in W^1_0(\omz)$ for any strictly increasing $(q-1)$-tuple $K$. Moreover, 
     \begin{equation}\label{12}
 	 \begin{aligned}
 	 Q_\Omega(f_N,f_N)=\sumprime_{|K|=q-1}\sum_{j=1}^n \int_\Omega \left|\frac{\partial f_N^K}{\partial \bar z_j}\right|^2\, dV=\dfrac{1}{4}\|\nabla f_N\|^2_\omz.
 	 \end{aligned}
 	 \end{equation}
 	 We then obtain (\ref{8}) by combining (\ref{11}) and (\ref{12}). 

     To establish \eqref{8a}, we note that
     \begin{equation}\label{12a}
     \begin{aligned}
     Q_\Omega(f_N,f_N)=\sumprime_{|K|=q-1}\sum_{j=1}^n\int_\Omega \left|\sum_{l=1}^n \frac{\partial d(z)}{\partial z_l}\frac{\partial f_{lK}}{\partial \bar z_j} +\sum_{l=1}^n \frac{\partial^2 d(z)}{\partial z_l\partial \bar z_j} f_{lK}\right|^2\, dV. 
     \end{aligned}
     \end{equation}
     The desirable inequality then follows from a simple use of the Cauchy-Schwarz inequality and the facts that $\sum_{l=1}^n |\partial d (z)/\partial z_l|^2=1/4$ near $\partial\Omega$ and 
     \[
     Q_\Omega(f, f)\ge\sumprime_{|J|=q}\sum_{j=1}^n \int_\Omega \left|\frac{\partial f_J}{\partial \bar z_j}\right|^2\, dV=\frac{1}{q}\sumprime_{|K|=q-1}\sum_{j=1}^n \int_\Omega \left|\frac{\partial f_{jK}}{\partial \bar z_j}\right|^2\, dV.
     \]
\end{proof}

The following lemma is a direct consequence of Lemma~\ref{Hardy1} and a theorem of Davies
\cite[Theorem4]{Davies00}. We sketch the proof for the reader's convenience.
	 
 	 \begin{lem}\label{Hardy2} Let $\Omega$ be a smooth bounded pseudoconvex domain of finite type in $\C^n$.  If $f$ is an eigenform for the $\dbar$-Neumann Laplacian with associated eigenvalue $\lambda(\Omega)$, then there exist constants  $0<\alpha<1$ and $C>0$ such that
 	 	\begin{equation}\label{13}
 	 	\|f_N\|_{A_\delta}\le C\delta^{3/2}(1+\lambda(\omz))^{\frac{1}{2}[\frac{1}{2\alpha}]+\frac{3}{4}}\|f\|,
 	 	\end{equation} 
 	 	where $A_\delta=\{z\in\omz | \dist(z,\partial\omz)<\delta\}$.
 	 \end{lem}
  
  \begin{proof}  It follows from \cite[Theorem 4]{Davies00} that  
 	 \begin{equation}\label{9}
 	 \|f_N\|_{A_\delta}^2\leq C\delta^{3}\|(\Delta^D+A)f_N\|_\omz\|(\Delta^D+A)^{1/2}f_N\|_\omz,
 	 \end{equation}
 	 where $\Delta^D$ is the Dirichlet Laplacian, acting componentwise on $f_N$. Note
 	 that
 	 \[
 	 \|(\Delta^D+A)^{1/2}f_N\|^2=\|\nabla f_N\|^2+A\|f_N\|^2=4Q_\Omega(f_N, f_N)+A\|f_N\|^2.
 	 \]
 	 Thus by Lemma~\ref{Hardy1}, this term is dominated by a constant multiple of 
 	 $$
 	 Q_\Omega(f, f)+\|f\|^2=(1+\lambda(\Omega))\|f\|^2.
 	 $$
 	
 	 To estimate the term $\|(\Delta^D+A)f_N\|_\Omega$, we observe that 
 	 \begin{equation}
 	 \begin{aligned}
 	 \|\Delta^D f_N\|^2_\omz &= \|\nabla^2 f_N\|^2_\omz=\Bigg\|\nabla^2 \mathop{{\sum}'}\limits_{|K|=q-1}\left(\sum_{j=1}^{n}\frac{\partial d(z)}{\partial z_j}f_{jK}\right)d\bar{z}_K\Bigg\|^2_\omz\\
 	 &\le C\left(\|f\|^2+\|\nabla f\|^2+\|\lambda(\omz)f_N\|^2\right)\\
 	 &\le C (1+\lambda(\Omega))^{2\big[\frac{1}{2\alpha}\big]+2}\|f\|^2,
 	 \end{aligned}
 	 \end{equation}
 	 for some constant $C$ depending on the $C^3$-norm of the defining function. Here in the last inequality above we have used \eqref{28b}. Combining the above estimates, we then obtain the desired estimate \eqref{13}.
 \end{proof}

We are now in position to prove Theorem~\ref{prop10}.

\begin{proof}[Proof of Theorem~\ref{prop10}]		
	Let $f$ be a normalized $(0, q)$-eigenform of $\Box_\Omega$ associated with eigenvalue $\lambda(\omz)$. Since $\Omega$ is of finite type, $f\in C^\infty_{(0, q)}(\overline{\Omega})$.  Let $d(z)=\dist(z,\partial\omz)$ and let $\eta_{\delta_j}(z)=\chi\left(d(z)/\delta_j\right)$ where $\chi$ is a smooth function such that $\chi(t)=0$ if $t<1$, $\chi(t)=1$ if $t>2$, and $0\le \chi'(t)\le 1$. Then $|\nabla\eta_{\delta_j}|\le 1/\delta_j$, and $\supp\eta_{\delta_j}\subset\omz$ provided $\delta_j$ is sufficiently small.  Set $f_{\delta_j}(z)=\eta_{\delta_j}(z)\widetilde{f}(z)$. Then $\supp f_{\delta_j}\subset\subset\Omega_j$ and hence $f_{\delta_j}\in\Dom(\dbarstar_{\omz_j})$. 
	
	Let $\ec \colon W^s(\Omega)\to W^s(\C^n)$ be a continuous extension operator. Recall that the norm of this operator depends only
	on $n$, $s$, and the Lipschitz constant of $\Omega$ (\cite[Ch~VI.3, Theorem~5]{Stein70}).
	We have
	\begin{equation}\label{ec1}
	\|f_{\delta_j} -\ec f\|^2_{\Omega_j}\le \int_{A_{j, 3\delta_j}} |\ec f|^2\, dV\le C\delta_j,
	\end{equation}
	where, as before, $A_{j, 3\delta_j}=\{z\in\Omega_j \mid \dist(z, \partial\Omega_j)\le 3\delta_j\}$. We have
	$$\vartheta f_{\delta_j}=\eta_{\delta_j}\vartheta f+\delta_j^{-1}\chi'(d(z)/\delta_j) f_N,$$
	where
	$$
	f_N=(\dbar d(z))^*\lrcorner f= \sumprime_{|K|=q-1}f_N^K(z)\, d\bar{z}_K \quad \text{and}\quad 
	f_N^K(z)=\sum_{j=1}^{n}\frac{\partial d(z)}{\partial z_j}f_{jK}(z).
	$$ 
	Note that since $f\in \Dom(\dbarstar_{\omz})$, $f^K_N(z)=0$ on $\partial\Omega$.  It follows from Lemma~\ref{lem15} and Lemma~\ref{Hardy2} that
	\begin{align*}
	\|\vartheta f_{\delta_j}\|_{A_{2\delta_j}}\le \|\vartheta f\|_{A_{2\delta_j}}+\delta_j^{-1}\|f_N\|_{A_{2\delta_j}}
	\le C\delta^{1/2}_j 
	\end{align*}
	and
	\begin{equation}\label{52}
	\begin{aligned}
	&\|\vartheta f_{\delta_j}-\vartheta f\|_{\omz_j}\le \|(\eta_{\delta_j}-1)\vartheta f\|_{\omz_j}+\delta_j^{-1}\|\chi'(d/\delta_j)f_N\|_{\omz_j}\\
	&\qquad\le \|\vartheta f\|_{A_{2\delta_j}}+\delta_j^{-1}\|f_N\|_{A_{2\delta_j}}\le C\delta_j^{1/2}.
	\end{aligned}
	\end{equation}
	
	  Define
	$T_j\colon \Dom(\Box_\Omega)\to \Dom(Q_j)$  by
	\begin{equation}\label{ft3}
	T_j f=\dbarstar_j N_j\dbar_j \ec f+\dbar_j N_j\dbarstar_j f_{\delta_j}=\dbarstar_j\dbar_j N_j \ec f+\dbar_j \dbarstar_jN_j f_{\delta_j}.
	\end{equation}	
 Then
	\begin{equation}\label{f1}
	\begin{aligned}
	\|T_j f\|^2_{\Omega_j}&=\|\dbarstar_j \dbar_j N_j \ec f\|^2_{\Omega_j}+\|\dbar_j \dbarstar_j N_j  f_{\delta_j}\|^2_{\Omega_j} \\
	&=\|f_{\delta_j}\|^2_{\Omega_j}+\|\dbarstar_j\dbar_j N_j \ec f\|^2_{\Omega_j}-\|\dbarstar_j \dbar_j N_j  f_{\delta_j}\|^2_{\Omega_j}.
	\end{aligned}
	\end{equation}	
	Notice that
	\begin{equation}\label{f2}
	\begin{aligned}
	\|\dbarstar_j\dbar_j N_j f_{\delta_j}\|^2_{\Omega_j}&=\|\dbarstar_j \dbar_jN_j (f_{\delta_j}-\ec f)+\dbarstar_j\dbar_j N_j\ec f\|^2_{\Omega_j} \\
	&=\|\dbarstar_j\dbar_j N_j (f_{\delta_j}-\ec f)\|^2_{\Omega_j}+\|\dbarstar_j\dbar_j N_j \ec f\|^2_{\Omega_j}\\
	&\qquad +2\rea \langle \dbarstar_j \dbar_j N_j (f_{\delta_j}-\ec f), \, \dbarstar_j\dbar_j N_j \ec f\rangle_{\Omega_j}.\\
	\end{aligned}
	\end{equation}
	The first term on the right hand side above is estimated by
	\begin{equation}\label{ec1a}
\|\dbarstar_j\dbar_j N_j (f_{\delta_j}-\ec f)\|^2_{\Omega_j}\le \|f_{\delta_j}-\ec f\|^2_{\Omega_j}\le C\delta_j.
	\end{equation}
We also have
	\[
	\big|\langle \dbarstar_j \dbar_j N_j (f_{\delta_j}-\ec f), \, \dbarstar_j\dbar_j N_j \ec f\rangle_{\Omega_j}\big|=\big|\langle f_{\delta_j}-\ec f, \, \dbarstar_j\dbar_j N_j \ec f\rangle_{\Omega_j}\big|,
	\]
	which is estimated from above by
	\begin{equation}\label{ec2}
	\big|\langle f_{\delta_j}-\ec f, \, \dbarstar_j\dbar_j N_j \ec f\rangle_{\Omega_j}\big|\le C\|f_{\delta_j}-\ec f\|_{A_{j, 3\delta_j}}\|\dbarstar_j\dbar_j N_j \ec f\|_{A_{j, 3\delta_j}}.
	\end{equation}
When there is no finite type assumption on $\Omega_j$, since
\[
\|\dbarstar_j\dbar_j N_j \ec f\|_{A_{j, 3\delta_j}}\le \|\dbarstar_j\dbar_j N_j \ec f\|_{\Omega_j}\le \|\ec f\|_{\Omega_j}\le C,
\]
it follows from \eqref{ec1} and \eqref{ec2} that
\begin{equation}\label{ec3}
\big|\langle f_{\delta_j}-\ec f, \, \dbarstar_j\dbar_j N_j \ec f\rangle_{\Omega_j}\big|
\le C\delta^{1/2}_j.
\end{equation}
It then follows from \eqref{f1}, \eqref{f2}, \eqref{ec1a}, and \eqref{ec3} that in this case, we have
\begin{equation}\label{ec4}
\left| \|T_jf\|^2_{\Omega_j} -\|f\|^2_{\Omega}\right| \le C\delta^{1/2}_j.
\end{equation}

Under the uniform finite type assumption on  $\Omega_j$, by  Catlin's subelliptic estimate, $N_j \ec f$ is smooth and its $C^2$-norm is bounded from above by a constant independent of $j$ (see Lemma~\ref{lem15} and the subsequent remark above). It follows that 
\begin{equation}\label{ec5}
	\big|\langle f_{\delta_j}-\ec f, \, \dbarstar_j\dbar_j N_j \ec f\rangle_{\Omega_j}\big|
	\le C\delta_j.
\end{equation}
Thus, in this case, we have
\begin{equation}\label{ec4b}
	\left| \|T_jf\|^2_{\Omega_j} -\|f\|^2_{\Omega}\right| \le C\delta_j.
\end{equation}

On the one hand, since
	\[
	\|\dbarstar_j T f\|_{\Omega_j}^2=\|\dbarstar_j \dbar_j N_j \dbarstar_j f_{\delta_j}\|_{\omz_j}^2=\|\dbarstar_j f_{\delta_j}\|_{\omz_j}^2,
	\]
	we have
	\[
	\left| \|\dbarstar_j T_j f\|_{\omz_j}^2-\|\dbarstar f\|_{\omz}^2\right|
	\le \int_\Omega (1-\eta_{\delta_j}^2) |\vartheta f|^2+2|\eta_{\delta_j}\vartheta f|\cdot|\delta^{-1}_j \chi'(d(z)/\delta_j) f_N|+ \delta^{-2}_j |f_N|^2\, dV.
	\]
	By Lemma~\ref{Hardy2} and the fact that the $C^1$-norm of $f$ is bounded from above (see Lemma~\ref{lem15}), we have
	\begin{equation}\label{ec6}
	\left| \|\dbarstar_j T_j f\|_{\omz_j}^2-\|\dbarstar f\|_{\omz}^2\right| \le C\delta_j.
	\end{equation}
	On the other hand, since $\dbar_j T_j f=\dbar \ec f$, we have
	\begin{equation}\label{ec7}
	|\|\dbar_j T_jf\|_{\Omega_j}^2-\|\dbar f\|^2_\Omega|\le \int_{\Omega_j\setminus\Omega} |\dbar \ec f|^2\, dV+\int_{\Omega\setminus\Omega_j} |\dbar f|^2\, dV \le C\delta_j.
	\end{equation}
Thus
\begin{equation}\label{ec8}
|Q_{\Omega_j}(T_j f, T_j f)-Q_\Omega(f, f)|\le C\delta_j.	
\end{equation}
When \eqref{ec8} is coupled with \eqref{ec4}, we then obtain from Lemma~\ref{first} the inequality \eqref{ft-1}. When it is coupled with \eqref{ec4b}, we obtain the second inequality in \eqref{ft-2}. The first inequality of \eqref{ft-2} is proved similarly and is left to the interested reader.
\end{proof}

\begin{remark} The power of $\delta_j$ in \eqref{ft-2} is sharp. For example, let $\B$ be the unit ball in $\C^n$. Then $\lambda_k^q(r_j\B)=\lambda_k^q(\B)/r^2_j$. Thus
	\[
	\big|\lambda_k^q(r_j\B)-\lambda_k^q (\B)\big|=(1-r_j)(1+r_j) \lambda_k^q(\B)/r^2_j\approx \delta_j
	\]
	as $\delta_j=1-r_j\to 0$.
\end{remark}
\section{Resolvent convergence}\label{sec:re}

Let $\Omega$ be a bounded domain in $\C^n$. We consider $L^2_{(0, q)}(\Omega)$ be a subspace of $L^2_{(0, q)}(\C^n)$ consisting of forms vanishing outside $\Omega$. For $\lambda\in\C\setminus\R$,  
we extend the resolvent operator $R_\lambda(\square_\Omega)=(\lambda I-\square_\Omega)^{-1}$ to act on $L^2_{(0,q)}(\C^n)$ by setting $R_\lambda(\square_\Omega)=0$ on $L^2_{(0, q)}(\C^n)\ominus L^2_{(0, q)}(\Omega)$.  Let $\Omega_j$ and $\Omega$ be bounded domains in $\C^n$, we say that $\square_{\Omega_j}$ converges to $\square_\Omega$ in strong (respectively in norm) resolvent sense if for all $\lambda\in\C\setminus\R$, $R_\lambda (\square_{\Omega_j})$ converges strongly (respectively in norm) to $R_\lambda(\square_\Omega)$ as operators acting on $L^2_{(0, q)}(\C^n)$. When $\Omega$ and $\omz_j$ are pseudoconvex, we will extend $N_\Omega=\square^{-1}_\Omega$ and $N_{\Omega_j}=\square^{-1}_{\Omega_j}$  to act on $L^2_{(0, q)}(\C^n)$ in a likewise manner.
Since the spectra of $\Box_{\Omega_j}$ and $\Box_{\Omega}$ are uniformly bounded away from $0$, it is easy to see that $\Box_{\Omega_j}$ converges to $\Box_{\Omega}$ in strong resolvent sense if $N _j$ converges to $N$ strongly on $L^2_{(0, q)}(\C^n)$ (see, e.g., \cite[Theorem VIII.9]{ReedSimon80}). Here, as before, to economize the notation, we write $N_{\omz_j}$ and $N_\omz$ simply as $N_j$ and $N$ respectively.

\begin{thm}\label{th:re}
   $\omz$ is a bounded pseudoconvex domain in $\C^n$ with $C^1$ boundary. Let $\{\omz_j\}_{j\in\nb}$ be a sequence of bounded pseudoconvex domain such that $\delta_j=d_H(\omz,\omz_j)\to0$ as $j\to\infty$. Then $\Box_{\omz_j}$ converge to $\Box_\omz$ in strong resolvent sense.
\end{thm}

\begin{proof}
	The proof follows the same lines of arguments as in Theorem~\ref{prop5}. 
	Let $f\in L^2_{(0,q)}(\C^n)$ and let $g=Nf$. Then $f_\omz:=f|_\omz=\Box_\omz  g=\dbar\dbarstar g+\dbarstar\dbar g$. Since $ g,\dbar g\in \Dom(\dbarstar_\Omega)$, it follows from the minimality of $\dbarstar$ that for any $0<\varepsilon<1$,  there exist $\phi\in\dc_{(0,q)}(\omz)$ and $\varphi\in\dc_{(0,q+1)}(\omz)$ such that 
	\begin{equation}\label{re-1}
	\|\phi- g\|_\omz+\|\dbarstar(\phi- g)\|_\omz+\|\varphi-\dbar g\|_\omz+\|\dbarstar(\varphi- \dbar g)\|_\omz<\varepsilon. 
	\end{equation}	
	By choosing $j$ sufficiently large, we have $\phi\in\dc_{(0,q)}(\omz_{j})$ and  $\varphi\in\dc_{(0,q+1)}(\omz_{j})$. 
	
	As in the proof of Theorem~\ref{prop5}, we will use $\widetilde{F}$ to denote the extension of $F$ to $\C^n$ by letting $\widetilde F=0$ outside of $\omz$ and $\widehat{F}$ to denote the form constructed by \eqref{hat} with $\delta=\delta_j$. Thus $\widehat{\dbarstar g}$, $\widehat g\in \Dom(\dbar_j)$. Since 
	\begin{align*}
		\|N _jf-N  f\|_{\C^n}
		&= \|N _jf-N  f\|_{\omz_{j}}+\|Nf\|_{\omz\setminus\omz_j}\\
		&\le \|N_j\widetilde{f_\omz}-\widetilde{g}\|_{\omz_{j}}+\|N_j (f-\widetilde{f_\omz})\|_{\omz_j}+\|Nf\|_{\omz\setminus\omz_j}\\
		&\le \|N_j\widetilde{f_\omz}-\widetilde{g}\|_{\omz_{j}}+C\| f-\widetilde{f_\omz}\|_{\omz_j}+\|Nf\|_{\omz\setminus\omz_j}		
	\end{align*}	
	and the last two terms above goes to $0$ as $j\to \infty$, it suffices to prove that
	$\|N_j\widetilde{f_\omz}-\widetilde{g}\|_{\omz_{j}}\to 0$.
	We have
	\[
		\|N_j\widetilde{f_\omz}-\widetilde g\|_{\Omega_j}\le\|N_{j}\widetilde{\dbar\,\dbarstar g}-\dbar _j\dbarstar_jN_j\widetilde{g}\|_{\Omega_j}+\|N_{j}\dbarstar\widetilde{\dbar g}-\dbarstar_j\dbar_jN_j\widetilde{g}\|:=I+II.
	\]
	Note that	
		\begin{align*}
			I&= \|\dbar _jN _j\widehat{\dbarstar g}+N_j(\widetilde{\dbar\,\dbarstar g}-\dbar\,\widehat{\dbarstar g})-\dbar _j\dbarstar_jN _j\widetilde{g}\|_{\Omega_j}\\
			&=\|\dbar _jN _j\dbarstar \widetilde{g}+ \dbar _jN _j(\widehat{\dbarstar g}-\widetilde{\dbarstar g})+N_j(\widetilde{\dbar\,\dbarstar g}-\dbar\,\widehat{\dbarstar g})-\dbar _j\dbarstar_jN _j\widetilde{g}\|_{\Omega_j}\\
			&\le\|\dbar _j\dbarstar_jN _j(\phi-\widetilde{g})\|_{\Omega_j}+ \|\dbar _jN _j\dbarstar(\widetilde{g}-\phi)\|_{\Omega_j}\\
			&\qquad + \|\dbar _jN _j(\widehat{\dbarstar g}-\widetilde{\dbarstar g})\|_{\Omega_j}+\|N_j(\widetilde{\dbar\,\dbarstar g}-\dbar\,\widehat{\dbarstar g})\|_{\Omega_j}\\	
			&\le\|\phi-g\|_{\Omega}+ C\|\dbarstar(g-\phi)\|_{\Omega}+ C\|\widehat{\dbarstar g}-\widetilde{\dbarstar g}\|_{\Omega_j}+C\|\widetilde{\dbar\,\dbarstar g}-\dbar\,\widehat{\dbarstar g}\|_{\Omega_j}.\\
		\end{align*}
		It follows from \eqref{re-1} and Lemma~\ref{lm:hatcheck} that $I\to 0$ as $j\to\infty$.  Similarly, we have
		\begin{align*}
		  II&\le \|\dbarstar_j\dbar_jN_j(\widehat{g}-\widetilde{g})\|_{\Omega_j} +\|\dbarstar_jN_j(\widetilde{\dbar g}-\dbar \widehat{g})\|_{\Omega_j}\\
		  &\qquad +\|\dbarstar_jN_j(\varphi-\widetilde{\dbar g})\|_{\Omega_j}+\|N_{j}\dbarstar(\widetilde{\dbar  g}-\varphi)\|_{\Omega_j}\\
		  &\le \|\widehat{g}-\widetilde{g}\|_{\Omega_j} +C\|\widetilde{\dbar g}-\dbar \widehat{g}\|_{\Omega_j}+C\|\varphi-\dbar g\|_{\Omega}+C\|\dbarstar(\dbar  g-\varphi)\|_{\Omega},\\
			\end{align*}
			which again goes to $0$ as $j\to\infty$. Thus $\|N_j\widetilde{f_\omz}-\widetilde{g}\|_{\omz_{j}}\to 0$. \end{proof}

\begin{remark} $(1)$ As for Theorem~\ref{prop1}, Theorem~\ref{th:re} holds without the $C^1$ assumption on $\partial\omz$ if the $\omz_j$'s are contained in $\omz$. $(2)$ One cannot expect that $\Box_{\Omega_j}$ converges to $\Box_\Omega$ in norm resolvent sense. For example, let $\Omega_j$ and $\Omega$ be bounded convex domains in $\C^n$ such that $\Omega$ is exhausted by $\Omega_j$ from inside.  Suppose  $\partial\Omega$ contains a complex analytic variety but $\partial\Omega_j$ does not. Then $N_j$ is compact but $N$ is not $($see \cite{FuStraube98}$)$. Thus in this case, 
$N_j$ does not converges to $N$ in norm.  
\end{remark}

\bigskip

\noindent{\bf Acknowledgments.} Part of this work was done during the visits of the first author to Princeton University, the University of Notre Dame, and Xiamen University, and the second to Rutgers University-Camden, and both to the Erwin Schr\"{o}dinger Institute. The authors
thank these institutions for hospitality. The second author also thanks Professor Chunhui Qiu for his kind encouragement and support. The first author was supported in part by a grant from the National Science Foundation (DMS-1500952). The second author was supported in part by a grant from the National Natural Science Foundation of China (Grant No.~11571288). 



\enddocument

\end